\documentclass[11pt]{amsart}

\usepackage{array,float,verbatim,graphics,enumitem, fancyhdr, color, soul}
\usepackage{amssymb, amsmath, amsthm, amsbsy}
\usepackage[normalem]{ulem}

\numberwithin{equation}{section}
\numberwithin{figure}{section}

\theoremstyle{plain}
 \newtheorem{theorem}{Theorem}[section]
 \newtheorem{proposition}[theorem]{Proposition}
 \newtheorem{lemma}[theorem]{Lemma}
 \newtheorem{corollary}[theorem]{Corollary}
 \newtheorem{conjecture}[theorem]{Conjecture}
 \newtheorem{thmABC}{Theorem}

\theoremstyle{definition}
 \newtheorem{definition}[theorem]{Definition}

\theoremstyle{remark}
 \newtheorem{remark}[theorem]{Remark}
 
 \newtheorem*{acknowledgements}{Acknowledgements}
 \newtheorem{question}[theorem]{Question}

\newcommand{\mfg}{\mathfrak{g}}
\newcommand{\mfp}{\mathfrak{p}}
\newcommand{\mfP}{\mathfrak{P}}
\newcommand{\mfz}{\mathfrak{z}}
\newcommand{\eps}{\ensuremath{\epsilon}}
\newcommand{\nega}{{\rm neg}}
\newcommand{\udrl}{\underline}
\newcommand{\C}{\ensuremath{\mathbb{C}}}

\newcommand{\Fq}{\ensuremath{\mathbb{F}_q}}
\newcommand{\N}{\ensuremath{\mathbb{N}}}
\newcommand{\Q}{\ensuremath{\mathbb{Q}}}
\newcommand{\R}{\ensuremath{\mathbb{R}}}
\newcommand{\mcL}{\ensuremath{\mathcal{L}}}
\newcommand{\T}{\ensuremath{\mathcal{T}}}

\newcommand{\Z}{\ensuremath{\mathbb{Z}}}

\newcommand{\bfa}{\ensuremath{\mathbf{a}}}
\newcommand{\bfb}{\ensuremath{\mathbf{b}}}
\newcommand{\bfc}{\ensuremath{\mathbf{c}}}
\newcommand{\bfe}{\ensuremath{\mathbf{e}}}
\newcommand{\bff}{\ensuremath{\mathbf{f}}}

\newcommand{\bfr}{\ensuremath{\mathbf{r}}}
\newcommand{\bfx}{\ensuremath{\mathbf{x}}}
\newcommand{\bfy}{\ensuremath{\mathbf{y}}}
\newcommand{\bfz}{\ensuremath{\mathbf{z}}}

\newcommand{\bfG}{{\bf G}}
\newcommand{\bfH}{{\bf H}}

\newcommand{\bfY}{\ensuremath{\mathbf{Y}}}

\newcommand{\calF}{\ensuremath{\mathcal{F}}}
\newcommand{\calG}{\ensuremath{\mathcal{G}}}
\newcommand{\calH}{\ensuremath{\mathcal{H}}}
\newcommand{\calI}{\ensuremath{\mathcal{I}}}

\newcommand{\calN}{\ensuremath{\mathcal{N}}}

\newcommand{\calO}{\ensuremath{\mathcal{O}}}
\newcommand{\calP}{\ensuremath{\mathcal{P}}}
\newcommand{\calR}{\ensuremath{\mathcal{R}}}
\newcommand{\calT}{\ensuremath{\mathcal{T}}}
\newcommand{\calS}{\ensuremath{\mathcal{S}}}

\newcommand{\calZ}{\ensuremath{\mathcal{Z}}}

\newcommand{\rk}{\ensuremath{{\rm rk}}}

\newcommand{\ol}{\overline}

\newcommand{\Gri}{\ensuremath{\mathcal{O}}}

\newcommand{\Gfi}{\ensuremath{K}}
\newcommand{\dotcup}{\ensuremath{\mathbin{\mathaccent\cdot\cup}}}

\newcommand{\lri}{\mathfrak{o}}
\newcommand{\Lri}{\mathfrak{O}}
\newcommand{\lfi}{\mathfrak{k}}

\newcommand{\gp}[1]{\mathrm{gp}(#1)}
\newcommand{\rmN}{\mathrm{N}}

\renewcommand{\epsilon}{\varepsilon}
\renewcommand{\phi}{\varphi}

\DeclareMathOperator{\maj}{maj}

\DeclareMathOperator{\Rad}{Rad}

\DeclareMathOperator{\GL}{GL}

\DeclareMathOperator{\Mat}{Mat}

\DeclareMathOperator{\Stab}{Stab}
\DeclareMathOperator{\real}{Re}
\DeclareMathOperator{\Alt}{Alt}
\DeclareMathOperator{\Sym}{Sym}
\DeclareMathOperator{\Hom}{Hom}
\DeclareMathOperator{\Ind}{Ind}
\DeclareMathOperator{\Res}{Res}
\DeclareMathOperator{\ad}{ad}

\DeclareMathOperator{\diag}{diag}
\DeclareMathOperator{\Pfaff}{Pf}
\DeclareMathOperator{\Irr}{Irr}

\DeclareMathOperator{\coxinv}{inv}
\DeclareMathOperator{\coxneg}{neg}

\DeclareMathOperator{\rmaj}{rmaj} 

\def \tud {\textup{d}}

\begin{document}

\title[Representation zeta functions of nilpotent
  groups]{Representation zeta functions of nilpotent groups and
  generating functions for Weyl groups of type~$B$}

\author{Alexander Stasinski and Christopher Voll}

\address{School of Mathematics, University of Southampton, University
  Road, Sou\-thampton SO17 1BJ, United Kingdom}
\curraddr{A.~Stasinski: Department of Mathematical Sciences, Durham
  University, Durham DH1 3LE, United Kingdom. C.~Voll: Fakult\"at
  f\"ur Mathematik, Universit\"at Bielefeld\\ Postfach
  100131\\ D-33501 Bielefeld\\Germany}
\email{alexander.stasinski@durham.ac.uk, C.Voll.98@cantab.net}

\keywords{finitely generated nilpotent groups, representation zeta
  functions, Kirillov orbit method, signed permutation statistics,
  $q$-series, Igusa local zeta functions, prehomogeneous vector spaces}
\subjclass[2000]{22E55, 20F69, 05A15, 11M41}

\begin{abstract}
We study representation zeta functions of finitely generated,
torsion-free nilpotent groups which are groups of rational points of
unipotent group schemes over rings of integers of number fields. Using
the Kirillov orbit method and $\mfp$-adic integration, we prove
rationality and functional equations for almost all local factors of
the Euler products of these zeta functions.

We further give explicit formulae, in terms of Dedekind zeta
functions, for the zeta functions of class-$2$-nilpotent groups
obtained from three infinite families of group schemes, generalising
the integral Heisenberg group. As an immediate corollary, we obtain
precise asymptotics for the representation growth of these groups, and
key analytic properties of their zeta functions, such as meromorphic
continuation.

We express the local factors of these zeta functions in terms of
generating functions for finite Weyl groups of type~$B$. This allows
us to establish a formula for the joint distribution of three
functions, or `statistics', on such Weyl groups.

Finally, we compare our explicit formulae to $\mfp$-adic integrals
associated to relative invariants of three infinite families of
prehomogeneous vector spaces.
\end{abstract}

\maketitle

\setcounter{tocdepth}{1}
\thispagestyle{empty}

\section{Introduction and statement of main results}\label{sec:intro}

\subsection{Background and summary}
Let $G$
be a group and denote, for $n\in\N$, by $r_{n}(G)$ the number of
isomorphism classes of $n$-dimensional irreducible complex
representations of~$G$. The group $G$ is called
(\emph{representation}) \emph{rigid} if $r_{n}(G)$ is finite for
all~$n\in\N$. If $G$ is rigid and the numbers $r_{n}(G)$ grow at most
polynomially, a fruitful approach to the study of these numbers is to
encode them into a Dirichlet generating function, which is called the
representation zeta function of $G$. A variety of tools from complex
analysis, algebraic geometry, model theory and combinatorics is
available to investigate these zeta functions, and thus to shed light
on the arithmetic and asymptotic properties of the sequence~
$(r_{n}(G))$; see, for instance,
\cite{LarsenLubotzky/08,cr_AKOV/09,BartholdidelaHarpe/10}. Throughout
this paper we use the term `representations' to refer to complex
representations. In the context of topological groups, we only
consider continuous representations.

In the current paper we study representation zeta functions associated
to finitely generated, torsion-free nilpotent groups (so-called
$\calT$-groups).  Such groups are not rigid. Indeed, a non-trivial
$\calT$-group has infinitely many representations of dimension
$1$. However, $\calT$-groups are {}``rigid up to twisting'' by
$1$-dimensional representations.  More precisely, let $G$ be a group
and let $\rho$ and $\sigma$ be irreducible representations of~$G$. One
calls $\rho$ and $\sigma$ \emph{twist-equivalent} if there exists a
$1$-dimensional representation $\chi$ of $G$ such
that~$\rho\cong\chi\otimes\sigma$. If $G$ is a topological group, we
demand in addition that $\chi$ be continuous. This defines an
equivalence relation on the set of irreducible representations of $G$,
whose classes are called \emph{twist-isoclasses}. For a group $G$ and
$n\in\N$ we denote by $\widetilde{r}_{n}(G)$ the number of
twist-isoclasses of $G$ of dimension~$n$. It is known that if $G$ is a
$\calT$-group and $n\in\N$ then there exists a finite quotient $G(n)$
of $G$ such that every $n$-dimensional representation of $G$ is
twist-equivalent to one that factors through $G(n)$. In particular,
the number $\widetilde{r}_{n}(G)$ is finite for all~$n\in\N$; see
\cite[Theorem 6.6]{LubotzkyMagid/85}.  The \emph{representation zeta
function} of a $\calT$-group $G$ is defined to be the Dirichlet
generating function \begin{equation}
\zeta_{G}(s):=\sum_{n=1}^{\infty}\widetilde{r}_n(G)n^{-s},\label{def:zeta}\end{equation}
where $s$ is a complex variable. The sequence $(\widetilde{r}_n(G))$
grows polynomially, and thus this series converges on a complex right
half-plane $\{s\in\C\mid\real(s)>\alpha\}$, for some~$\alpha\in\R$;
see Lemma~\ref{lem:poly}. The \emph{abscissa of convergence
$\alpha(G)$} of $\zeta_{G}(s)$, that is, the infimum of such~$\alpha$,
gives the precise degree of polynomial growth.  More precisely, if $G$
is non-trivial, $\alpha(G)$ is the smallest value such that
$\sum_{n=1}^{N}\widetilde{r}_n(G)=O(N^{\alpha(G)+\varepsilon})$ for
every $\varepsilon\in\R_{>0}$.

The zeta function $\zeta_G(s)$ has an Euler product, indexed by the
rational primes. Indeed, one has
\begin{equation}\label{equ:coarse euler}
 \zeta_{G}(s)=\prod_{p\text{ prime}}\zeta_{G,p}(s),
\end{equation}
where
$\zeta_{G,p}(s):=\sum_{i=0}^{\infty}\widetilde{r}_{p^{i}}(G)p^{-is}$,
for each prime number $p$.  This Euler product simply reflects the
facts that every representation of $G$ is twist-equivalent to one that
factors through a finite quotient, and that finite nilpotent groups
are direct products of their Sylow $p$-subgroups. Much deeper lies the
fact, proved by Hrushovski and Martin, that all the factors
in~\eqref{equ:coarse euler} are rational functions in the parameter
$p^{-s}$; see~\cite[Theorem 8.4]{HrushovskiMartin/07}. Another deep
result establishes functional equations of almost all of the local
factors `upon inversion of $p$'; see~\cite[Theorem D]{Voll/10}.

The only explicitly computed examples of representation zeta functions
of $\calT$-groups in print prior to the current paper are formulae for
the zeta function of the Heisenberg group $\bfH(\mathcal{O})$ of
upper-unitriangular $3\times3$-matrices over the ring of integers
$\mathcal{O}$ of a number field $K$ of degree at most~$2$;
cf.~\cite{NunleyMagid/89} for $K=\mathbb{Q}$, and
\cite[Theorem~1.1]{Ezzat/14} for $K$ a quadratic number
field. These examples agree with the formula
\begin{equation}\label{equ:heisenberg number field}
\zeta_{\bfH(\mathcal{O})}(s)=\frac{\zeta_{K}(s-1)}{\zeta_{K}(s)} = \prod_{\mfp}
\frac{1-q^{-s}}{1-q^{1-s}},
\end{equation}
where $\zeta_{K}(s)$ is the Dedekind zeta function of $K$, $\mfp$
ranges over the non-zero prime ideals of $\Gri$, and $q=|\Gri/\mfp|$.
In particular, we have $\zeta_{\bfH(\Z)}(s) = \sum_{n=1}^\infty
\phi(n) n^{-s}$, where $\phi$ is the Euler totient function. Ezzat
conjectured in~\cite{Ezzat/14} that~\eqref{equ:heisenberg
  number field} holds for arbitrary number fields $K$. This is in fact
implied by one of our main results; cf.~Theorem~\ref{thmABC:thm
  B}. Note that for $G=\bfH(\Gri)$ the Euler product
in~\eqref{equ:heisenberg number field} is finer than the
product~\eqref{equ:coarse euler}, and that all the local factors are
rational in~$q^{-s}$. We will show that these facts, too, are special
cases of more general phenomena.

The $\calT$-groups studied in the current paper are obtained from
unipotent group schemes over rings of integers of number fields. From
now on, let $\mathcal{O}$ be the ring of integers of a number
field~$K$. Let $\bfG$ be a unipotent group scheme defined over~$\Gri$;
see Section~\ref{subsec:group schemes}. Then $\bfG(\mathcal{O})$ is a
$\calT$-group. Given a non-zero prime ideal $\mfp$ of~$\Gri$, we
denote by $\Gri_\mfp$ the completion of $\Gri$ at $\mfp$, and by $p$
the residue field characteristic of~$\Gri_\mfp$. In
Proposition~\ref{pro:Fine Euler prod} we establish the Euler
factorisation
\begin{equation}\label{equ:euler}
 \zeta_{\bfG(\Gri)}(s)=\prod_{\mfp}\zeta_{\bfG(\Gri_\mfp)}(s),
\end{equation}
indexed by the non-zero prime ideals of~$\Gri$, with local factors
given by \[
\zeta_{\bfG(\Gri_{\mfp})}(s)=\sum_{i=0}^{\infty}\widetilde{r}_{p^{i}}(\bfG(\Gri_{\mfp}))p^{-is}.\]
The factorisation~\eqref{equ:euler} reflects the fact that unipotent
groups have the Congruence Subgroup Property and the strong
approximation property. Note that it refines~\eqref{equ:coarse euler}.

The Euler product~\eqref{equ:euler} is similar to the Euler product
satisfied by representation zeta functions of arithmetic subgroups of
semisimple algebraic groups;
cf.~\cite[Proposition~1.3]{LarsenLubotzky/08}. The archimedean factors
present in the latter are absent in the realm of nilpotent groups,
reflecting the fact that every representation of a $\T$-group is
twist-equivalent to one which factors over a finite quotient. The
non-archimedean factors in the context of semisimple groups have been
studied using techniques from model theory and $\mfp$-adic
integration; see, for instance,~\cite{Jaikin/06, Avni/08, cr_AKOV/09}.

\bigskip
The following summarises the main results of the paper and some of its
motivation. On the one hand, we set up a general framework for
studying representation zeta functions of $\calT$-groups obtained from
group schemes over $\calO$, which in turn are associated to
$\calO$-Lie lattices. The principal tools here are the Kirillov orbit
method and $\mfp$-adic integration, which enable us to analyse the
local factors of Euler products of the form~\eqref{equ:euler}. In
particular, we derive formulae which are uniform under extensions of
$\Gri$, and prove rationality and functional equations for almost all
local factors; see Theorem~\ref{thmABC:thm A}.

On the other hand, we study three infinite families of $\calT$-groups
of nilpotency class~$2$ for which we carry out this general analysis
explicitly and prove more precise results. Our approach allows us to
compute the zeta functions of these groups as finite products
involving translates of Dedekind zeta functions; see
Theorem~\ref{thmABC:thm B}. This shows, in particular, that in these
cases all of the local factors in~\eqref{equ:euler} are rational both
in $q$ and in~$q^{-s}$, and satisfy the functional equations which
Theorem~\ref{thmABC:thm A} asserts only for almost all local
factors. Using our formulae it is easy to read off the zeta functions'
key analytic properties, such as abscissae of convergence, meromorphic
continuation, and location and order of the poles; see
Corollary~\ref{cor:thm B}. The formulae for the zeta functions also
imply precise asymptotic formulae for the representation growth of the
relevant groups.

We do not expect the strong regularity properties displayed by the
zeta functions in Theorem~\ref{thmABC:thm B} to hold in general, not
even in nilpotency class~$2$. Nevertheless, the interesting problem
arises to characterise the groups for which they do hold.

Originally our interest in the three specific families of $\T$-groups
studied in Theorem~\ref{thmABC:thm B} arose out of an analogy with
$\mfp$-adic integrals associated to reduced irreducible prehomogeneous
vector spaces with relative invariants; see
Section~\ref{sec:pvs}. These are complex vector spaces on which
algebraic groups act with Zariski dense orbits. Three infinite
families of such prehomogeneous vector spaces are given by $n\times
n$-matrices, symmetric matrices and antisymmetric matrices,
respectively. The relative invariants in these cases are the
determinant or, in the case of antisymmetric matrices, the Pfaffian;
the associated $\mfp$-adic integrals are Igusa's local zeta functions
associated to these polynomials. These well-known integrals are of
particular interest as they are cases in which the Bernstein-Sato
polynomial conjecture is known to hold. This conjecture connects the
real parts of the poles of Igusa's local zeta function with the zeros
of the integrand's Bernstein-Sato polynomial. We record in this paper
an intriguing analogy between the representation zeta functions
computed in Theorem~\ref{thmABC:thm B} and the $\mfp$-adic integrals
associated to the above-mentioned prehomogeneous vector spaces. In
particular, the local pole spectra of the former are obtained from the
pole spectra of the latter by a simple translation by integers which
turn out to be the (global) abscissae of convergence of the relevant
zeta functions. Our observations give rise to the general question to
what degree the local pole spectra of zeta functions of groups reflect
geometric invariants like zeros of Bernstein-Sato polynomials. This
points to a possible analogue of the Bernstein-Sato polynomial
conjecture for zeta functions of groups.

Our explicit `multiplicative' formulae given in
Theorem~\ref{thmABC:thm B} are proved using `additive' formulae given
in Theorem~\ref{thmABC:thm C}. The bridge between the two is given by
an identity which is reminiscent of the $q$-multinomial theorem in the
theory of basic hypergeometric series; see
Proposition~\ref{pro:multinomial}. The present paper provides a
natural motivation for this identity in the context of zeta functions
of groups; see also Remark~\ref{rem:subgroup zeta}.  We show that the
polynomials appearing in Theorem~\ref{thmABC:thm C} have a rich
combinatorial structure: we express them as generating functions for
statistics on finite Weyl groups of type~$B$, also known as
hyperoctahedral groups. As an application, we prove a formula for the
joint distribution of three statistics on Weyl groups of type~$B$; see
Proposition~\ref{pro:distribution}. It would be interesting to know
whether other types of Weyl groups occur in this framework, and to
study the resulting analogues of our results.

\subsection{Uniformity, rationality and functional equations}\label{subsec:uniratfunc}
In Section~\ref{subsec:group schemes} we describe a class of group
schemes defined by nilpotent $\Gri$-Lie lattices. By an $\Gri$-Lie
lattice we mean a free and finitely generated $\Gri$-module, together
with an antisymmetric, bi-additive form satisfying the Jacobi
identity. Let $\Lambda$ be a nilpotent $\Gri$-Lie lattice of
nilpotency class $c$, and write $\Lambda'$ for the derived Lie
lattice~$[\Lambda,\Lambda]$. If $\Lambda$ satisfies the condition
$\Lambda'\subseteq c!\Lambda$ it gives rise to a unipotent group
scheme $\bfG_{\Lambda}$ over $\Gri$ via the Hausdorff series. When the residue characteristic $p$ is odd or when $p=2$ and $c\geq4$, there exists a Kirillov orbit method to describe the
irreducible representations of groups of the
form~$\bfG_{\Lambda}(\Gri_{\mfp})$ in terms of co-adjoint orbits; see
Section~\ref{subsec:kirillov}. In the case where $c=2$ we give an
unconditional construction of a unipotent group scheme associated to
$\Lambda$ which coincides with $\bfG_\Lambda$ if~$\Lambda'\subseteq
2\Lambda$, and describe a Kirillov orbit method for
$\bfG_{\Lambda}(\Gri_\mfp)$ which holds for all primes; see
Section~\ref{subsec:class 2}. In any case, whenever the Kirillov orbit
method applies it allows us to describe local representation zeta
functions in terms of Poincar\'e series and suitable $\mfp$-adic
integrals. The first main result of the paper establishes universal
formulae for the generic factors in Euler products of the
form~\eqref{equ:euler} for groups of the form $\bfG_\Lambda(\Gri_L)$,
where $\Gri_L$ is the ring of integers in a finite extension $L$ of
the number field~$K$. Let $d=\dim_K(\Lambda'\otimes_\Gri K)$.

\begin{thmABC} \label{thmABC:thm A}
 There exist a finite set $S$ of prime ideals of $\Gri$, $t\in\N$, and
 a rational function
 $R(X_{1},\dots,X_{t},Y)\in\Q(X_{1},\dots,X_{t},Y)$ such that, for
 every prime ideal $\mfp$ of $\Gri$ with $\mfp\not\in S$, the
 following is true. There exist algebraic integers
 $\lambda_{1},\dots,\lambda_{t}$, depending on~$\mfp$, such that, for
 all finite extensions $\Lri$ of $\lri=\Gri_{\mfp}$ one has \[
 \zeta_{\bfG_{\Lambda}(\Lri)}(s) =
 R(\lambda_{1}^{f},\dots,\lambda_{t}^{f},q^{-fs}),\] where $q$ denotes
 the residue field cardinality of $\lri=\Gri_\mfp$, and
 $f=f(\Lri,\lri)$ is the relative degree of inertia. In particular,
 the local factor $\zeta_{\bfG_{\Lambda}(\Lri)}(s)$ is a rational
 function in~$q^{-fs}$. Furthermore, the following functional equation
 holds:
 \begin{equation}
   \zeta_{\bfG_{\Lambda}(\Lri)}(s)|_{\substack{q\rightarrow
       q^{-1}\\ \lambda_{i}\rightarrow\lambda_{i}^{-1}}
       }=q^{fd}\zeta_{\bfG_{\Lambda}(\Lri)}(s).\label{equ:funeq}
 \end{equation}
\end{thmABC}

\begin{remark}
 The statement of Theorem~\ref{thmABC:thm A} is analogous
 to~\cite[Theorem~A]{AKOVI/10}. In fact, the proof of~\ref{thmABC:thm
   A} leans heavily on the methods developed in~\cite{AKOVI/10}. We
 note, however, that~\cite[Theorem~A]{AKOVI/10} applies only to
 certain pro-$p$ subgroups of the groups featuring in the Euler
 factors, whereas Theorem~\ref{thmABC:thm A} yields a formula for
 almost all factors in~\eqref{equ:euler}. The functional
 equation~\eqref{equ:funeq} refines the statement of~\cite[Theorem
   D]{Voll/10}.
\end{remark}
The proof of Theorem~\ref{thmABC:thm A} is found in Section~\ref{subsec:proof thm A}.
\subsection{Groups of types $F$, $G$ and $H$, and multiplicative formulae}
Much of the present paper is concerned with the representation zeta
functions of $\T$-groups obtained from three specific infinite
families of group schemes. These arise from class-$2$-nilpotent $\Z$-Lie
lattices, and each family generalises a different aspect
of the Heisenberg group scheme~$\bfH$.

\begin{definition}\label{def:lattices}
Let $n\in\N$ and $\delta\in\{0,1\}$.  We define the following
nilpotent $\Z$-Lie lattices of class $2$:
\begin{align*}
 \calF_{n,\delta} & =\langle x_k, y_{ij} \mid[x_k,y_{ij}],
 \,[x_{i},x_{j}]-y_{ij},\; 1\leq k \leq 2n+\delta, 1\leq i<j\leq
 2n+\delta\rangle,\\
\calG_{n} & =\langle x_k,y_{ij}\mid[x_k,y_{ij}],\, [x_{i},x_{n+j}] -
y_{ij}, \; 1\leq k \leq 2n, 1\leq i,j\leq n\rangle, \\
\calH_{n} &
=\langle x_k,y_{ij} \mid[x_k,y_{ij}],[x_{i},x_{n+j}]-y_{ij},
[x_{j},x_{n+i}]-y_{ij}, \; 1 \leq k \leq 2n,1\leq i\leq j\leq
n\rangle.
\end{align*}
Note that all these Lie lattices are isomorphic to quotients of the
free class-$2$-nilpotent Lie rings generated by the $x_k$. In fact,
$\calF_{n,\delta}$ is isomorphic to the free class-$2$-nilpotent Lie
ring on $x_1,\dots,x_{2n+\delta}$. In any case, the elements
$x_k,y_{ij}$ yield $\Z$-bases for the respective Lie lattices.
\end{definition}

Let $F_{n,\delta}$, $G_{n}$ and $H_{n}$ denote the unipotent group
schemes over $\mathbb{Z}$ associated to the Lie lattices
$\calF_{n,\delta}$, $\calG_{n}$ and~$\calH_{n}$, respectively. We
call groups of the form $F_{n,\delta}(\Gri)$, $G_{n}(\Gri)$ and
$H_{n}(\Gri)$ groups of type $F$, $G$ and $H$, respectively.  The
groups $F_{n,\delta}(\mathbb{Z})$ are the free class-$2$-nilpotent
groups on $2n+\delta$ generators. Note that
$F_{1,0}(\Gri)=G_{1}(\Gri)=H_{1}(\Gri)=\bfH(\Gri)$, the Heisenberg
group over~$\Gri$.

Apart from being natural generalisations of $\bfH(\Gri)$, the groups
of type $F$, $G$ and $H$ are of interest as their zeta functions are close analogues of zeta
integrals associated to relative invariants of irreducible
prehomogeneous vector spaces. This connection is explored in
Section~\ref{sec:pvs}. In our second main result we give explicit
formulae for the zeta functions of groups of type $F$, $G$ and $H$, in
terms of Dedekind zeta functions. For $n\in\N$, we set $m=\lfloor n/2
\rfloor$ and~$\eps=n-2m\in\{0,1\}$, so that $n=2m+\eps$.

\begin{thmABC}\label{thmABC:thm B}
 Let $n=2m+\eps\in\N$, $\delta\in\{0,1\}$ and let $K$ be a number field with
 ring of integers~$\Gri$. Then
 \begin{align} \zeta_{F_{n,\delta}(\Gri)}(s) &
  =\prod_{i=0}^{n-1}\frac{\zeta_{K}(s-2(n+i+\delta)+1)}{\zeta_{K}(s-2i)},\label{equ:mult
    F}\\ \zeta_{G_{n}(\Gri)}(s) &
  =\prod_{i=0}^{n-1}\frac{\zeta_{K}(s-n-i)}{\zeta_{K}(s-i)},\label{equ:mult
    G}\\ \zeta_{H_{n}(\Gri)}(s) &
  =\frac{\zeta_{K}(s-n)}{\zeta_{K}(s)}\prod_{i=0}^{m-1}\frac{\zeta_{K}(2(s-m-i-\varepsilon)-1)}{\zeta_{K}(2(s-i-1))}.\label{equ:mult
    H}
 \end{align}
\end{thmABC}

\begin{corollary}\label{cor:thm B}
 Let $\mathbf{G}\in\{F_{n,\delta},G_{n},H_{n}\}$.
 \begin{enumerate}
  \item \label{item:funeqs}For all non-zero prime ideals $\mfp$ of $\Gri$, the following functional equation holds:
   \begin{equation}\label{equ:funeqs thm B}
    \zeta_{\bfG(\Gri_\mfp)}(s)|_{q\rightarrow q^{-1}} = q^{d(\bfG)}\zeta_{\bfG(\Gri_\mfp)}(s),
   \end{equation}
   where
   $$d(F_{n,\delta})=\binom{2n+\delta}{2}, \quad d(G_{n})= n^2, \quad
   d(H_{n})= \binom{n+1}{2}$$ is the $\Z$-rank of the corresponding
   derived Lie lattice.
  \item \label{item:alphas}The abscissa of convergence
    $\alpha(\mathbf{G}(\Gri))$ of
    $\zeta_{\mathbf{G}(\Gri)}(s)$ is an integer. More precisely, we
    have
    \begin{equation}
     \alpha(F_{n,\delta}(\Gri))=2(2n+\delta-1),\quad\alpha(G_{n}(\Gri))=2n,\quad\alpha(H_{n}(\Gri))=n+1.\label{equ:alphas}
    \end{equation}
  In particular, $\alpha(\bfG) := \alpha(\mathbf{G}(\Gri))$ is
  independent of $\Gri$.
 \item \label{item:analytic}The zeta function $\zeta_{\mathbf{G}(\Gri)}(s)$ has
   meromorphic continuation to the whole complex plane. The continued
   zeta function has no singularities on the line
   $\{s\in\C\mid\real(s)=\alpha(\mathbf{G})\}$, apart from a simple
   pole at $s=\alpha(\mathbf{G})$.
 \item \label{item:asymptotic}There exists a constant
   $c(\mathbf{G}(\Gri))$, given explicitly in terms of special values
   of the Dedekind zeta function~$\zeta_K(s)$, such that
\[
\sum_{n\leq N}\widetilde{r}_n(\mathbf{G}(\Gri))\sim c(\mathbf{G}(\Gri))\cdot
N^{\alpha(\mathbf{G})}\quad\text{ as }N\longrightarrow\infty.\]

\end{enumerate}
\end{corollary}
Theorem~\ref{thmABC:thm B} and its corollary are proved in Section~\ref{sec:proof thm B}.

We remark that the functional equations~\eqref{equ:funeqs thm B}
illustrate~\eqref{equ:funeq}, which Theorem~\ref{thmABC:thm A} only
asserts for almost all~$\mfp$. It is of interest to what extent the
assertions of Corollary~\ref{cor:thm B} generalise to more general
group schemes. Specifically, we ask the following.
\begin{question}\label{qun:analytic}
 Let $\bfG$ be a unipotent group scheme defined over the ring of
 integers $\Gri$ of a number field $\Gfi$, and let $L$ be a finite
 extension of $K$, with ring of integers~$\Gri_L$. Is the abscissa of
 convergence $\alpha(\bfG(\Gri_L))$ independent of~$L$? Is it a
 rational number?  Does the zeta function $\zeta_{\bfG(\Gri)}(s)$
 allow for analytic continuation beyond its abscissa of convergence?
\end{question}
The last two questions would have an affirmative answer if local
representation zeta functions of $\calT$-groups were `cone integrals'
in the sense of~\cite{duSG/00}.

\subsection{Additive formulae and Weyl group generating functions}
In our third main result we prove `additive formulae' for the local
factors of the zeta functions of groups of type $F$, $G$ and~$H$.  To
state this result we introduce some notation. Throughout the paper,
$X$, $Y$ and $Z$ will denote indeterminates in the
field~$\Q(X,Y,Z)$. For $N\in\N$, we set $[N] = \{1,\dots,N\}$ and
$[N]_0 = \{0,1,\dots,N\}$. We write $(\udrl{N})_{X}$ for the
polynomial~$1-X^{N}$. We set $(\underline{0})_{X} = 1$ and write
$(\underline{N})_X!$ for
$(\underline{1})_{X}(\underline{2})_{X}\cdots(\underline{N})_{X}$.
Given $n\in\N$ and $I=\{i_{1},\dots,i_{l}\}_{<}\subseteq[n-1]_{0}$, we
set $i_{0} = 0$ and $i_{l+1} = n$, respectively. Here the subscript
$<$ on $\{i_{1},\dots,i_{l}\}$ indicates that~$i_1<\dots<i_l$.  Note
that $i_1=n$ when $I=\varnothing$. For $j\in[l]_{0}$ we define
$\mu_{j}=i_{j+1}-i_{j}$. For $a,b\in\N_{0}$ such that $a\geq b$, we
have the `$X$-binomial coefficient', also known as the Gaussian
polynomial \[
\binom{a}{b}_{X}=\frac{(\underline{a})_{X}!}{(\udrl{a-b})_{X}!(\underline{b})_{X}!}.\]
Furthermore, we have the `$X$-multinomial coefficient' \[
\binom{n}{I}_{X}=\binom{n}{i_{l}}_{X}\binom{i_{l}}{i_{l-1}}_{X}\cdots\binom{i_{2}}{i_{1}}_{X}.\]
We also define the $Y$-Pochhammer symbol, or $Y$-shifted factorial, as
\begin{equation}\label{def:pochhammer}
(X;Y)_{n}=\prod_{i=0}^{n-1}(1-XY^{i}).
\end{equation}
In the literature, the above symbols are often defined in terms of a
formal variable $q$, and thus one often encounters the $q$-binomial
and $q$-multinomial coefficients, and $q$-Pochhammer symbol,
respectively. In our context, however, $q$ is always a prime power, so
we choose $X$ and $Y$ as formal variables.

Given a fixed prime power $q=p^{f}$ we write $(\underline{N})$ for
$(\underline{N})_{q^{-1}}$.
\begin{thmABC} \label{thmABC:thm C}
 Let $n\in\N$, $\delta\in\{0,1\}$ and let $K$ be a number field with
 ring of integers~$\Gri$.  Let
 $\mathbf{G}\in\{F_{n,\delta},G_{n},H_{n}\}$. There exist polynomials
 $f_{\mathbf{G},I}(X)\in\Z[X]$, ${I\subseteq[n-1]_{0}}$, and natural
 numbers $a(\mathbf{G},i)$, ${i\in[n-1]_{0}}$, such that, for all
 non-zero prime ideals $\mfp$ of~$\Gri$, one has
\begin{equation}
 \zeta_{\mathbf{G}(\Gri_\mfp)}(s)=\sum_{I\subseteq[n-1]_{0}}f_{\mathbf{G},I}(q^{-1})\prod_{i\in
   I}\frac{q^{a(\mathbf{G},i)-(n-i)s}}{1-q^{a(\mathbf{G},i)-(n-i)s}},\label{main
   theorem formula}
\end{equation}
where $q=|\Gri/\mfp|$.  The
data $f_{\mathbf{G},I}(X)$ and $a(\mathbf{G},i)$ are given in the
following table.
\[
\begin{array}{c|c|c} \mathbf{G} & f_{\mathbf{G},I}(X) &
  a(\mathbf{G},i)\\ \hline F_{n,\delta} &
  \binom{n}{I}_{X^2}(X^{2(i_1+\delta)+1};X^2)_{n-i_{1}} &
  \binom{2n+\delta}{2}-\binom{2i+\delta}{2}\\ G_{n} &
  \binom{n}{I}_{X}(X^{i_1+1};X)_{n-i_1} & n^{2}-i^{2}\\ H_{n} &
  \left(\prod_{j=1}^{l}(X^2;X^2)_{\lfloor\mu_{j}/2\rfloor}^{-1}\right)(X^{i_1+1};X)_{n-i_1}
  & \binom{n+1}{2}-\binom{i+1}{2}
\end{array}
\]
\end{thmABC}
Note that $f_{\mathbf{G},I}(X)$ and $a(\mathbf{G},i)$ are independent
of $\Gri$ and~$\mfp$. Theorem~\ref{thmABC:thm C} is proved in
Section~\ref{sec:proofThmC}.

The proof of the `multiplicative' Theorem~\ref{thmABC:thm B} relies on
the `additive' Theorem~\ref{thmABC:thm C}, together with the following
identity, which we prove in
Section~\ref{subsec:binomial}.

\begin{proposition} \label{pro:multinomial}
 For $n\in\mathbb{N}$ we have
\begin{gather}
\sum_{I\subseteq[n-1]_0}\binom{n}{I}_{X^{-1}}(YX^{-i_1-1};X^{-1})_{n-i_1}
  \prod_{i\in
  I}\gp{(X^{i}Z)^{n-i}} =\frac{(X^{-n}YZ;X)_n}{(Z;X)_n}.\label{equ:binomial
  B}
\end{gather}
\end{proposition}

We call the identity ~\eqref{equ:binomial B} of `multinomial type' due
to its analogy with the multinomial theorem. For groups of type $F$
and $G$, Theorem~\ref{thmABC:thm B} is a formal consequence of
Theorem~\ref{thmABC:thm C} and
Proposition~\ref{pro:multinomial}. Further combinatorial arguments are
needed to deal with groups of type~$H$, and these are treated in
Section~\ref{subsec:thm B}.

Proposition~\ref{pro:multinomial} has applications to generating
functions of statistics on Weyl groups of type $B$; for details, see
Section~\ref{subsec:weyl}. Let $n\in\N$, and consider the group $B_n$
of permutations $w$ of the set $[\pm n]_0:=\{-n,\dots,n\}$ such that,
for all $i\in[\pm n]_0$, $w(-i)=-w(i)$. We may identify $B_n$ with the
group of `signed permutation matrices', that is, monomial matrices
whose non-zero entries are in~$\{1,-1\}$. Interesting statistics on
$B_n$ include the usual Coxeter length function $l$ with respect to
the standard Coxeter generating set $S=\{s_{0},\dots,s_{n-1}\}$ and
the statistic `$\nega$' which keeps track of the number of negative
entries of a signed permutation. A result of Reiner in
\cite{Reiner-Signed-perm} allows us to express the polynomials
$f_{F_{n,\delta},I}(X)$ and $f_{G_n,I}(X)$ given in
Theorem~\ref{thmABC:thm C} in terms of the joint distribution of the
statistics $l$ and $\nega$ over descent classes in~$B_n$. For groups
of type $H$ we present a conjectural formula of this kind. To state
it, we introduce a new statistic $L$ on $B_n$. For $w\in B_{n}$ we
define
\begin{equation}\label{def:L}
 {L}(w)=\frac{1}{2}\#\{(x,y)\in[\pm n]_{0}^{2}\mid
 x<y,\ w(x)>w(y),\ x\not\equiv y\bmod{(2)}\}\in\N_0
\end{equation}
and write $D(w)=\{s\in S \mid l(ws)<l(w)\}$ for the (right) descent
set of~$w$. From now on, we identify $S$ with $[n-1]_0$ in the
obvious way. For $I\subseteq S$, let $I^c=[n-1]_0\setminus I$ denote the complement of $I$, and let $B_n^{I^c}=\{w\in B_n \mid D(w)\subseteq I\}$.
\begin{conjecture}\label{con:L}
\label{con:L-tilde}For $n\in\N$ and $I\subseteq[n-1]_{0}$
we have
\begin{equation*}
f_{H_{n},I}(X)=\sum_{w\in B_n^{I^c}}(-1)^{l(w)}X^{{L}(w)}.
\end{equation*}

\end{conjecture}
In a \cite{StasinskiVoll/13} we prove Conjecture~\ref{con:L} for
arbitrary $n\in\N$ and $I\in\left\{\{0\},[n-1]_{0}\right\}$, as well
as in the case where $n$ is even and
$I\subseteq[n-1]_{0}\cap2\N_0$. We remark that the statistic $L$ is a
natural extension of a statistic on the symmetric group $S_n$ defined
by Klopsch and the second author; cf.~\cite[Lemma
  5.2]{KlopschVoll/09}.

Combining our Weyl-group theoretical interpretations of the polynomials
$f_{\bfG,I}(X)$ with Proposition~\ref{pro:multinomial} allows us to describe the joint
distribution of three statistics on Weyl groups of type~$B$, namely $\sigma-l$, $\coxneg$ and $\rmaj$. The statistics $l$ and $\coxneg$ have already been introduced. We now give the definitions of $\sigma$ and $\rmaj$. For a general finite Weyl group~$W$, with root system $\Phi$ and simple
roots $\alpha_0,\dots,\alpha_{n-1}$, let $b_0,\dots,b_{n-1}$ denote the simple
root coordinates for half the sum of all positive roots, that is
$\sum_{\alpha\in\Phi}\alpha=2\sum_{i=0}^{n-1}b_i\alpha_i$.
Specifically, for $B_n$ with generating set $S$ as above, the simple
root coordinates $b_i$, $i\in[n-1]_0$ are given by $b_i=n^2-i^2$;
cf.~\cite[Plate~II]{Bourbaki/02} (note that we use the reverse
ordering of the simple roots). For $w\in B_n$, let
$$\sigma(w) = \sum_{i\in D(w)}b_i=\sum_{i\in D(w)}(n^2-i^2)$$ and define the
\emph{reverse major index} by
$\rmaj(w) = \sum_{i\in D(w)} (n-i)$.
\begin{proposition}\label{pro:distribution}
For $n\in\N$ we have
\begin{gather*}
\sum_{w\in B_n} X^{(\sigma-l)(w)}Y^{\nega(w)}Z^{\rmaj(w)} =
\prod_{i=0}^{n-1}\frac{(1+X^iYZ)(1-(X^{n+i}Z)^{n-i})}{1-X^{n+i}Z}.
\end{gather*}
\end{proposition}

The generating function for the statistic $\sigma-l$ over a general
finite Weyl group was studied in~\cite{StembridgeWaugh/98}. The
geometric relevance of the statistic $\sigma-l$ is explained
in~\cite[Lemma 2.2]{StembridgeWaugh/98}. Proposition~\ref{pro:distribution}
generalises \cite[Theorem 1.1]{StembridgeWaugh/98} in the case of Weyl
groups of type~$B$. A similar result, pertaining to the statistic $L$
defined in~\eqref{def:L}, exploits the two different expressions of
the local zeta functions of groups of type $H$ given by Theorems
\ref{thmABC:thm B} and~\ref{thmABC:thm C}; see
Proposition~\ref{pro:distribution L}.

\subsection{Notation}\label{subsec:Notation}
We record some of our recurrent notation. Throughout, we denote by $K$
a number field with ring of integers $\Gri=\Gri_K$. We denote by
$\mfp$ a non-zero prime ideal of $\Gri$, and sometimes write $\lri$
for the completion $\Gri_\mfp$ of $\Gri$ at~$\mfp$. We write $q$ for
the residue field cardinality $|\lri/\mfp|$ and $p$ for the residue
field characteristic of~$\lri$. The Dedekind zeta function
$\zeta_K(s)$ of $K$ is
\begin{equation}\label{equ:dedekind}
 \zeta_K(s) = \sum_{I\triangleleft\Gri}|\Gri:I|^{-s} =
\prod_{\mfp}\zeta_{K,\mfp}(s),
\end{equation}
where the product is indexed by the non-zero prime ideals of~$\Gri$,
and $\zeta_{K,\mfp}(s) = 1 / (1-q^{-s})$. For a non-trivial
$\lri$-module $M$, we write $M^*:=M\setminus \mfp M$, and for the trivial
$\lri$-module $\{0\}$ we set $\{0\}^*=\{0\}$. Given a ring $R$, we
write $\rk_R(M)$ to denote the rank of a free $R$-module $M$. We also
write $\rk(x)$ for the rank of a matrix~$x$. For a fixed $d\in\N$, we
write $W(\lri)=(\lri^d)^*$ and, given $N\in\N$, we set
$W_N(\lri)=\left((\lri/\mfp^N)^d\right)^*$.

For any compact abelian group $\mathfrak{a}$ we write
$\widehat{\mathfrak{a}}$ for its Pontryagin dual
$\Hom_\Z^{\textup{cont}}(\mathfrak{a},\mathbb{C}^\times)$. In the context
of nilpotent groups we use the notation $\widehat{G}$ to denote the
profinite completion of the $\T$-group~$G$. We write $\Irr(G)$ for the
collection of isomorphism classes of continuous, irreducible complex
representations of a topological group~$G$.

Given a subset $I\subseteq\N$ we write $I_0$ for $I\cup\{0\}$. For
$a,b\in\Z$, we use the notation $aI+b=\{ai+b \mid i\in I\}$. Given a
term $X$ different from~$1$, we often write $\gp{X}$ for the `geometric
progression'~$X/(1-X)$.

\section{Representation zeta functions of nilpotent groups}
In this section we develop some general machinery to study
representation zeta functions of finitely generated nilpotent groups
obtained from unipotent group schemes associated to nilpotent Lie
lattices over rings of integers of number fields. We first give a
general description of the construction of $\T$-groups from group
schemes, the Kirillov orbit method and tools from $\mfp$-adic
integration in Sections~\ref{subsec:group schemes}
and~\ref{subsec:kirillov}. We prove Theorem~\ref{thmABC:thm A} in
Section~\ref{subsec:proof thm A}. Section~\ref{subsec:class 2} is
dedicated to a more explicit analysis in the case of nilpotency class
$2$, affording slightly stronger results in this case. The results in
this section also prepare the ground for the subsequent proof of
Theorem~\ref{thmABC:thm C}.

\medskip
Let $G$ be a $\calT$-group. Recall that, for $n\in\N$, we denote by
$\widetilde{r}_n(G)$ the number of twist-isoclasses of $n$-dimensional irreducible
representations of~$G$. The following lemma establishes that the
sequence $(\widetilde{r}_n(G))$ has polynomial growth, so that the Dirichlet
series $\zeta_G(s)$ has non-empty domain of convergence.

\begin{lemma}\label{lem:poly}
 The series $\zeta_G(s)=\sum_{n=1}^\infty \widetilde{r}_n(G)n^{-s}$ converges on a
 complex half-plane.
\end{lemma}

\begin{proof}
 We need to show that the sequence $(\widetilde{r}_n(G))$ is bounded by a
 polynomial in~$n$. It is well-known that every finite-dimensional
 irreducible representation of a $\T$-group is monomial, that is induced
 from a $1$-dimensional representation of some subgroup. $\T$-groups
 are further known to have polynomial subgroup growth, that is the
 sequence of the numbers $a_n(G)$ of subgroups of $G$ of index $n$ is
 bounded by a polynomial in~$n$; see, for instance,
 \cite[Theorem~5.1]{LubotzkySegal/03}. It thus suffices to show that
 the sequence of the numbers of twist-isoclasses of representations of
 $G$ obtained by inducing to $G$ a $1$-dimensional representation of
 an index-$n$-subgroup of $G$ is bounded by a polynomial in~$n$. This
 follows from the fact that, given a subgroup $H$ of $G$ of index $n$,
 a $1$-dimensional representation $\chi$ of $G$ and a $1$-dimensional
 representation $\psi$ of $H$, we have that
 \begin{equation}\label{equ:induced}
  \chi\otimes \Ind_H^G(\psi) = \Ind_H^G(\Res_H^G(\chi)\otimes \psi),
 \end{equation}
 and that the index $|G'\cap H:H'|$ is bounded by a polynomial
 in~$n$. To see the latter note that, since $G$ is nilpotent, each
 subgroup of index $n$ in $G$ contains~$G^n$, and that the index of
 $(G^n)'$ in $G'$ is bounded by a power of~$n$, which only depends on
 the number of generators and the nilpotency class of~$G$.
\end{proof}

\subsection{Unipotent group schemes, $\T$-groups and nilpotent Lie lattices}\label{subsec:group schemes}

\subsubsection{$\T$-groups from group schemes}
Let $\mathbf{G}$ be an affine smooth group scheme over~$\Gri$, the
ring of integers of a number field~$\Gfi$.  We say that $\mathbf{G}$
is \emph{unipotent} (\emph{over $\Gri$}) if the geometric fibre
$\mathbf{G}\times_{\Gri}\overline{k(\mfp)}$ is a connected unipotent
algebraic group for all $\mfp\in\text{Spec}({\Gri})$.  Here $k(\mfp)$
is the residue field $\Gri/\mfp$ when $\mfp\neq(0)$, and
$k((0))=K$. Note that, if $\mathbf{G}$ is unipotent, it is
automatically finitely presented over $\Gri$, by the definition of smoothness; see
e.g.~\cite[Section~2]{Alex-RedGreenAlg}.  It is well-known that if $k$
is any field and $\mathbf{G}$ is unipotent over~$k$ then
$\mathbf{G}(k)$ embeds as a subgroup of the group of
upper-unitriangular matrices in $\GL_{N}(k)$, for some~$N$. Thus
$\mathbf{G}(\Gri)$ is nilpotent and torsion-free, by virtue of being a
subgroup of $\mathbf{G}(K$). Moreover, since $\mathbf{G}$ is affine
and finitely presented over $\Gri$, and $\Gri$ is free of finite rank
over $\mathbb{Z}$, the Weil restriction
$\text{Res}_{\Gri/\mathbb{Z}}\mathbf{G}$ is an affine finitely
presented group scheme over $\mathbb{Z}$;
cf.~\cite[Proposition~4.4]{Real-Etale-coho}.  By a result of Borel and
Harish-Chandra, the group of $\mathbb{Z}$-points of an affine group
scheme of finite type over $\mathbb{Z}$ is finitely generated;
cf.~\cite[Theorem~6.12]{Borel-Harish-Chandra}. Therefore
$\mathbf{G}(\Gri)=(\text{Res}_{\Gri/\mathbb{Z}}\mathbf{G})(\mathbb{Z})$
is finitely generated, and thus a $\calT$-group.

For a non-zero prime ideal $\mfp$ of $\Gri$, we denote by
$\Gri_{\mfp}$ the completion of $\Gri$ at~$\mfp$, with maximal ideal
$\mfp$, residue field cardinality $q$ and residue field
characteristic~$p$. Let $\bfG$ be a unipotent group scheme
over~$\Gri$. By the Congruence Subgroup Property for unipotent groups
(see, for instance, \cite{Chahal/80}), and the strong approximation
property for unipotent groups
(cf.~\cite[Lemma~5.5]{PlatonovRapinchuk/94}), the profinite completion
of the $\T$-group $\bfG(\Gri)$ satisfies
\begin{equation}
\widehat{\bfG(\Gri)}=\prod_{\mfp}\bfG(\Gri_{\mfp}),\label{equ:strong}\end{equation}
where $\mfp$ ranges over the non-zero prime ideals of $\Gri$. The
factorisation \eqref{equ:strong} implies that the zeta function
$\zeta_{\bfG(\Gri)}(s)$ satisfies an Euler product, indexed by the
non-zero prime ideals on~$\Gri$. Indeed, every finite-dimensional
irreducible complex representation of $\bfG(\Gri)$ is twist-equivalent
to one with finite image; see \cite[Theorem~6.6]{LubotzkyMagid/85}. We
thus have a dimension-preserving bijection between the
twist-isoclasses of representations of $\bfG(\Gri)$ on the one hand
and continuous irreducible complex representations of
$\widehat{\bfG(\Gri)}$ up to twists by continuous $1$-dimensional
representations on the other. Owing to the product \eqref{equ:strong}
and the resulting fact that \[
\Hom^{\textup{cts}}(\widehat{\bfG(\Gri)},\C^{\times})=\prod_{\mfp}\Hom^{\textup{cts}}(\widehat{\bfG(\Gri_{\mfp})},\C^{\times}),\]
the zeta function $\zeta_{\bfG(\Gri)}(s)$ therefore is the Euler
product of the local zeta functions
\begin{equation}
  \zeta_{\bfG(\Gri_{\mfp})}(s):=\sum_{i=0}^{\infty}\widetilde{r}_{p^{i}}(\bfG(\Gri_{\mfp}))p^{-is}.\label{euler
    factor}
\end{equation}
Note that $\widetilde{r}_n(\bfG(\Gri_{\mfp}))=0$ unless $n$ is a power of $p$, as
$\bfG(\Gri_{\mfp})$ is a pro-$p$ group. The above discussion is summarised in the following result.

\begin{proposition} \label{pro:Fine Euler prod}~
We have the Euler product
 \begin{equation*}
 \zeta_{\bfG(\Gri)}(s)=\prod_{\mfp}\zeta_{\bfG(\Gri_{\mfp})}(s).
 \end{equation*}
\end{proposition}

\begin{remark}
  We will show that the local factor $\zeta_{\bfG(\Gri_{\mfp})}(s)$ is
  in fact a rational function in $q^{-s}$ whenever the group scheme
  $\bfG$ is obtained from a Lie lattice in the sense defined in
  Section~\ref{subsubsec:group schemes} and $p$ is odd or $p=2$ and
  $c\neq3$; cf.\ Corollaries~\ref{cor:zeta=integral}
  and~\ref{cor:power series in q class 2}. We do not know whether
  these conditions are also necessary. In any case, it follows from
  work of Hrushovski and Martin that, for all rational primes $p$, the
  `mini Euler product' $\zeta_{\bfG(\Gri),p}(s)=\prod_{\mfp\vert
    p}\zeta_{\bfG(\Gri_{\mfp})}(s)$ is rational in~$p^{-s}$;
  see~\cite[Theorem 8.4]{HrushovskiMartin/07}.
\end{remark}
\subsubsection{Group schemes from Lie lattices}\label{subsubsec:group schemes}

Recall the notion of Lie lattice from Section~\ref{subsec:uniratfunc}.  Let $(\Lambda,[\,\cdot\,,\,\cdot\,])$ be a nilpotent $\Gri$-Lie
lattice of $\Gri$-rank $h$ and nilpotency class $c$. Choose an
$\Gri$-basis $(x_{1},\dots,x_{h})$ for $\Lambda$.  For any
$\Gri$-algebra~$R$, let $\Lambda(R):=\Lambda\otimes_{\Gri}R$.  Then
$(x_{1}\otimes1,\dots,x_{h}\otimes1)$ is an $R$-basis for
$\Lambda(R)$; cf.~\cite[XVI, Proposition~2.3]{Lang-Algebra}.  We write
$\Lambda'$ for the derived Lie lattice~$[\Lambda,\Lambda]$.

Assume that $\Lambda'\subseteq c!\Lambda$. By means of the Hausdorff
series one may define a group structure on $\Lambda(R)$ by setting,
for $x,y\in\Lambda(R)$, \begin{align*} x*y &
=x+y+\frac{1}{2}[x,y]+\frac{1}{12}[[x,y],y]+\cdots,\\ x^{-1} &
=-x;\end{align*} see~\cite[Chapter 9.2]{Khukhro/98}. The group
$(\Lambda(R),*)$ is nilpotent of class~$c$. In co-ordinates with
respect to the $R$-basis $(x_1\otimes 1,\dots,x_h\otimes 1)$ for
$\Lambda(R)$, the group operations are given by polynomials
over~$\Gri$ which are independent of~$R$. This defines a unipotent
group scheme $\bfG_{\Lambda}$ over~$\Gri$, isomorphic as a scheme to
affine $h$-space over $\Gri$, representing the group functor\[
R\longmapsto(\Lambda(R),*).\]

The group $\bfG_\Lambda(\Gri)$ is a $\T$-group of nilpotency class
$c$ and Hirsch length
$h(\bfG_\Lambda(\Gri))=\rk_{\Z}(\Gri)h=[K:\mathbb{Q}]h$.  If $R$ is a
finitely generated pro-$p$ ring, such as~$\Gri_{\mfp}$, the group
$\bfG_\Lambda(R)$ is a finitely generated class-$c$-nilpotent pro-$p$ group.

\begin{remark}
Let $G$ be a $\T$-group of nilpotency class~$c$. It is well-known that
there exists a $\Q$-Lie algebra $\mcL_G(\Q)$ of $\Q$-dimension $h(G)$,
and an injective mapping $\log:G\rightarrow\mcL_G(\Q)$ such that
$\log(G)$ spans $\mcL_G(\Q)$ over~$\Q$; cf., for instance,
\cite[Chapter 6]{Segal/83}. It is further known that there exists a
subgroup $H$ of $G$ of finite index such that $\log(H)$ is a $\Z$-Lie
lattice inside $\mcL_G(\Q)$ and $\log(H)'\subseteq c!\log(H)$. Thus
$H$ may be recovered as the group of $\Z$-points of the group scheme
defined by $\log(H)$, and it makes sense to study the $\Gri$-points of
this group scheme for extensions $\Gri$ of~$\Z$;
cf. Remark~\ref{rem:class 2} and \cite[Sections 1 and
5]{GSegal/84}. In particular, for all $p$ which do not divide the
index $|G:H|$, we have that~$\zeta_{G,p}(s)=\zeta_{H,p}(s)$.
\end{remark}

We close this section with a simple lemma which we will need in later
computations with coordinates. Let $Z(\Lambda) =
\{x\in\Lambda\mid[x,\Lambda]=0\}$ be the centre of~$\Lambda$. Let $R$
be either $\Gri$ or~$\lri$, and let $M$ be a finitely generated
$R$-module, and $N$ an $R$-submodule of $M$. We write $\iota(N)$ for
the \emph{isolator} of $N$ in $M$, that is the smallest submodule $L$ of $M$
containing $N$ such that $M/L$ is torsion-free. We say that $N$ is
\emph{isolated} in $M$ if $\iota(N)=N$, that is if $M/N$ is torsion-free. Note that if $M$ is torsion-free then $N$ is isolated in $M$ if and only if $N$ is a pure submodule of $M$.

\begin{lemma}\label{lem:Free centre}
The centre $Z(\Lambda)$ is isolated in $\Lambda$. Moreover, suppose
that $M$ is a free $\Gri$-module of finite rank and $N$ an isolated
submodule of  $M$. Then there exists a free finite index submodule
$N_{0}$ of $N$ and a free finite index submodule $M_{0}$ of $M$
containing $N_{0}$ such that there exists a basis for $N_{0}$
which can be extended to a basis for $M_{0}$.\end{lemma}
\begin{proof}
  Let $x\in\Lambda$. If $x+Z(\Lambda)\in\Lambda/Z(\Lambda)$ is torsion
  then there exists a non-zero element $a\in\Gri$ such that $ax\in
  Z(\Lambda)$. This implies that $[ax,\Lambda]=a[x,\Lambda]=0$, but
  since $\Lambda$ is torsion-free, this means that $[x,\Lambda]=0$,
  that is, $x\in Z(\Lambda)$.  We now prove the second claim. Note
  that, if $\Gri$ is a principal ideal domain, the claim follows in a
  straightforward way from the structure theory of modules over such
  rings. In general, we use well-known facts from the structure theory
  of finitely generated modules over Dedekind domains; see, for
  instance,~\cite[Chapter~10.6]{Cohn/03}.  Since torsion-free modules
  over a Dedekind domain are
  projective~\cite[Proposition~10.6.6]{Cohn/03}, the short exact
  sequence
\[
0\longrightarrow N\longrightarrow M\longrightarrow M/N\longrightarrow0
\]
splits, and so $M\cong N\oplus M/N$. By
\cite[Theorem~10.6.11]{Cohn/03} there exist non-zero ideals $I_{1}$
and $I_{2}$ of $\Gri$ such that $N\cong\Gri^{r-1}\oplus I_{1}$ and
$M/N\cong\Gri^{s-1}\oplus I_{2}$, where $r$ and $s$ are the $K$-ranks
of $N$ and $M/N$, respectively.  The $K$-rank of $M$ is then $r+s$ and
$M\cong\Gri^{r+s}$. Thus $\Gri^{r+s}\cong\Gri^{r-1}\oplus
I_{1}\oplus\Gri^{s-1}\oplus I_{2}\cong\Gri^{r+s-2}\oplus I_{1}\oplus
I_{2}\cong\Gri^{r+s-1}\oplus I_{1}I_{2}$, where for the last
isomorphism we have used \cite[(10.6.8)]{Cohn/03}. Invoking
\cite[Theorem~10.6.11]{Cohn/03} again yields that $\Gri\cong
I_{1}I_{2}$, so in particular $I_{1}I_{2}$ is a free $\Gri$-module
which is contained in both $I_{1}$ and $I_{2}$.  Hence
$\Gri^{r-1}\oplus I_{1}I_{2}$ is a free submodule of $N$ and
$\Gri^{s-1}\oplus I_{1}I_{2}$ is a free submodule of $M/N$. Let
\[
N_{0}:=\Gri^{r-1}\oplus I_{1}I_{2}\quad\mathrm{and}\quad M_{0}:=N_{0}\oplus(\Gri^{s-1}\oplus I_{1}I_{2}).
\]
Then $N/N_{0}\cong I_{1}/I_{1}I_{2}$ and $M/M_{0}\cong N/N_{0}\oplus I_{2}/I_{1}I_{2}$
are finite as sets because every non-zero ideal of $\Gri$ is of finite
index. Any basis of $N_{0}$ can obviously be extended to $M_{0}$,
and the lemma is proved.\end{proof}

\subsection{Kirillov orbit method, Poincar\'e series and $\mfp$-adic integration}\label{subsec:kirillov}
Let $\Lambda$ be a nilpotent $\Gri$-Lie lattice such that $\Lambda'\subseteq c!\Lambda$, and let
$\bfG=\bfG_\Lambda$ be the unipotent group scheme over~$\Gri$ associated to $\Lambda$ as in Section~\ref{subsubsec:group schemes}. Our aim in the current
section is to provide tools to study and compute the local factors
$\zeta_{\bfG(\Gri_{\mfp})}(s)$.

\subsubsection{Kirillov orbit method}\label{subsubsec:kirillov}
A key technical tool in our analysis is the Kirillov orbit method for
groups of the form~$\bfG(\Gri_\mfp)$. Where it is applicable, it
provides a way to construct the irreducible representations of a group
in terms of co-adjoint orbits. For the class of $\T$-groups, this
method was pioneered by Howe in~\cite{Howe-nilpotent/77}. A treatment
of the Kirillov orbit method for pro-$p$ groups can be found
in~\cite{Gonzalez/09}.

We now fix a non-zero prime ideal $\mfp$ of $\Gri$ and write
$\lri=\Gri_\mfp$. Consider the $\lri$-Lie lattice
$\mfg:=\Lambda(\lri)=\Lambda\otimes_\Gri\lri$, and its Pontryagin dual
$\widehat{\mfg}=\Hom_\Z^\textup{cont}(\mfg,\C^\times)$. For any
$\psi\in\widehat{\mfg}$ we have an associated alternating bi-additive
form \[ B_{\psi}:\mfg\times
\mfg\longrightarrow\mathbb{C}^{\times},\quad(x,y)\longmapsto\psi([x,y]).\]
The form $B_{\psi}$ clearly only depends on the restriction of $\psi$
to $\mfg'$, and if $\phi\in\widehat{\mfg'}$ we simply write $B_{\phi}$
for $B_{\tilde{\phi}}$, where $\tilde{\phi}$ is any element in
$\widehat{\mfg}$ such that $\tilde{\phi}|_{\mfg'}=\phi$.  The radical
of the form $B_{\psi}$ is $\mathrm{Rad}(B_{\psi}):=\{x\in
\mfg\mid\forall y\in
\mfg:B_{\psi}(x,y)=1\}=\{x\in\mfg\mid\psi([x,\mfg])=1\}$.  The
following is a refinement of~\cite[Corollary~3.1]{Voll/10}.

\begin{theorem}\label{thm:zeta radical}
 If $p$ is odd or $p=2$ and $c\geq4$, we have
 \begin{equation}\label{equ:zeta}
 \zeta_{\bfG(\lri)}(s) = \sum_{\psi\in \widehat{\mfg'}} |
 \mfg:\Rad(B_\psi)|^{-s/2} | \mfg : \mfg_{\psi,2} |^{-1},
 \end{equation}
 where $\mfg_{\psi,2} = \{x \in \mfg |\;\psi([x,\mfg'])=1\}$.
\end{theorem}

\begin{proof}
 Assume first that $p>c$. The pro-$p$ group $\bfG(\lri)$ is then saturable and there
 exists a Kirillov correspondence between the finite co-adjoint orbits in the
 dual of the Lie algebra $\mfg$ and the continuous irreducible
 representations of $\bfG(\lri)$; cf.~\cite[Theorem~4.4]{Gonzalez/09}.

 Assume now that $p\leq c$. The condition $\Lambda'\subseteq
 c!\Lambda$ implies that $\mfg'\subseteq c!\mfg \subseteq p
 \mfg$. Furthermore, if $c\geq4$ the condition implies that
 $\mfg'\subseteq 4\mfg$. Applying the exponential map we obtain
 $\bfG(\lri)'\subseteq\bfG(\lri)^p$ and, if $c\geq4$,
 $\bfG(\lri)'\subseteq\bfG(\lri)^4$. Thus, if $p$ is odd or if $p=2$
 and $c\geq4$, we have that $\bfG(\lri)$ is a uniform pro-$p$
 group. For such groups there is again a Kirillov correspondence
 between the finite co-adjoint orbits in the dual of the Lie algebra
 $\mfg$ and the continuous irreducible representations of
 $\bfG(\lri)$; cf.~\cite[Section~2]{Jaikin/06}.

 Moreover, if $\Omega\subset\widehat{\mfg}$ is a finite co-adjoint orbit and
 $\psi\in\Omega$ then the dimension of the corresponding
 representation $\pi(\Omega)$ is given by
 $$\dim(\pi(\Omega)) =
 |\Omega|^{1/2}=|\bfG(\lri):\Stab_{\bfG(\lri)}(\psi)|^{1/2} = |\mfg :
 \Rad(B_{\psi})|^{1/2}.$$ The representation $\pi(\Omega)$ is obtained
 by inducing to $\bfG(\lri)$ the restriction of $\psi$ to a finite-index
 subgroup of~$\bfG(\lri)$. It follows that the twist-isoclass of the
 representation $\pi(\Omega)$ determines and is determined by the
 multiset of restrictions of the elements of $\Omega$ to~$\mfg'$. The
 number of distinct restrictions to $\mfg'$ in the orbit $\Omega$
 containing $\psi$ is $| \mfg : \mfg_{\psi,2} |$.
\end{proof}

\subsubsection{Poincar\'e series}\label{subsubsec:poincare}
In order to effectively compute the generating function
\eqref{equ:zeta} we express it in terms of Poincar\'e series. To this
end, we compute its terms in explicit coordinates. Write $\mfz$ for
the centre of~$\mfg$. By Lemma~\ref{lem:Free centre}, $\mfz$ is
isolated, that is, $\mfz=\iota(\mfz)$. Recall that $h =
\rk_\lri(\mfg)$ and set, in addition,
\begin{equation*}
d = \rk_\lri(\mfg'),\quad k =
\rk_\lri(\iota(\mfg')/\iota(\mfg'\cap\mfz))=
\rk_\lri(\iota(\mfg'+\mfz)/\mfz),\quad r-k =
\rk_\lri(\mfg/\iota(\mfg'+\mfz)),
\end{equation*}
so that $r=\rk_\lri(\mfg/\mfz)$. Note that the nilpotency class $c$ of
$\mfg$ is at most $2$ if and only if~$k=0$. We choose a
uniformiser~$\pi$ of~$\lri$, write $\overline{\phantom{x}}$ for the
natural surjection $\mfg\rightarrow\mfg/\mfz$, and choose an
$\lri$-basis
\begin{equation}\label{def:e}
  \bfe=(e_1,\dots,e_{r-k},\underbrace{e_{r-k+1},\dots,e_{r}}_{\iota(\overline{\mfg'+\mfz})},\overbrace{\underbrace{e_{r+1},\dots,e_{r-k+d}}_{\iota({\mfg'\cap\mfz})},e_{r-k+d+1},\dots,e_h}^{\mfz})
\end{equation}
for $\mfg$, as well as nonnegative integers $b_1,\dots,b_d$, such that
the following hold:
$$
\mfz = \langle e_{r+1},\dots,e_h \rangle_\lri
$$
\begin{align*}
\overline{\mfg'+\mfz} &=
  \langle
  \overline{\pi^{b_1}e_{r-k+1}},\dots,\overline{\pi^{b_k}e_{r}}\rangle_\lri&\iota(\overline{\mfg'+\mfz})
  &= \langle \overline{e_{r-k+1}},\dots,\overline{e_{r}}\rangle_\lri \\
\mfg'\cap\mfz &= \langle \pi^{b_{k+1}}e_{r+1},\dots,\pi^{b_d}e_{r-k+d}
  \rangle_\lri& \iota(\mfg'\cap\mfz) &= \langle
  e_{r+1},\dots,e_{r-k+d} \rangle_\lri.
\end{align*}
The existence of such integers is a consequence of the elementary
divisor theorem. We choose an $\lri$-basis
$\bff=(f_1,\dots,f_d)$
for $\mfg'$ such that
\begin{align*}
 (\ol{f_1},\dots,\ol{f_k}) &=
  (\ol{\pi^{b_1}e_{r-k+1}},\dots,\ol{\pi^{b_k}e_r}),\\ (f_{k+1},\dots,f_d)
    &= (\pi^{b_{k+1}}e_{r+1},\dots,\pi^{b_d}e_{r-k+d}).
\end{align*}
The structure constants $\lambda_{ij}^l\in\lri$ for $\mfg$, where
$i,j\in [r]$ and $l\in[d]$, with respect to these bases are defined by
$[e_i,e_j]=\sum_{l=1}^d\lambda_{ij}^lf_l$ and are encoded in the
\emph{commutator matrix}
$$\calR(\bfY) =
\left(\sum_{l=1}^d\lambda_{ij}^l
Y_l\right)_{ij}\in\Mat_{r}(\lri[\bfY]).$$ We define the submatrix
$$\calS(\bfY) = \left(\calR(\bfY)_{ij}\right)_{i\in[r],
  j\in\{r-k+1,\dots,r\}}\in \Mat_{r\times k}(\lri[\bfY]),$$ comprising
the last $k$ columns of $\calR(\bfY)$. As in \cite[Lemma 2.4]{AKOVI/10},
we write
$$\widehat{\mfg'} \cong \bigcup _{N\in\N_0} \Irr_N(\mfg'),$$ with
$\Irr_N(\mfg') = \widehat{\mfg'/\mfp^N\mfg'} \cong
\Hom_\lri(\mfg',\lri/\mfp^N)^*$. Let $N\in\N_0$. We say that
$w\in\Hom_\lri(\mfg',\lri)^*$ is a representative of $\psi \in
\Irr_N(\mfg')$ if $\psi$ is the image of $w$ under the natural
surjection $\Hom_\lri(\mfg',\lri)^* \rightarrow
\Hom_\lri(\mfg',\lri/\mfp^N)^*$. The $\lri$-basis ${\bf e}$ yields a
co-ordinate system $\overline{\mfg}=\mfg/\mfz \cong \lri^{r}, z
\mapsto \underline{z} = (z_1,\dots,z_{r})$. The dual basis
$\bff^\vee$, on the other hand, gives a co-ordinate system
$\Hom_\lri(\mfg',\lri)^*\cong (\lri^d)^*, w \mapsto
\underline{w}=(w_1,\dots,w_d)$. The following is proved in a way
similar to \cite[Lemma~3.3]{AKOVI/10}, and generalises the analysis
in~\cite[Section~3.4]{Voll/10}.
\begin{lemma}
 Let $w\in\Hom_\lri(\mfg',\lri)^*$ and $N\in\N_0$. Consider the
 element $\psi\in\Irr_N(\mfg')$ represented by $w$. Then, for every
 $z\in \mfg/\mfz$, we have
  \begin{align*}
   z \in \overline{\Rad(B_\psi)} &\Longleftrightarrow \underline{z} \cdot
   \calR(\underline{w}) \equiv 0 \mod \mfp^N\text{ and }\\ z \in
   \overline{\mfg_{\psi,2}} &\Longleftrightarrow \underline{z} \cdot
   \calS(\underline{w})\cdot\diag(\pi^{b_1},\dots,\pi^{b_k}) \equiv 0
   \mod \mfp^N.
  \end{align*}
\end{lemma}
We say that a matrix $S\in\Mat_{r\times k}(\lri)$ has (elementary
divisor) type $\bfc=(c_1,\dots,c_k)\in(\N_0\cup\{\infty\})^k$ --
written $\widetilde{\nu}(S)=\bfc$ -- if $S$ is equivalent by
elementary row and column operations to the $r\times k$-matrix
$$\left(\begin{matrix}\pi^{c_1}&&\\&\ddots&\\&&\pi^{c_k}\\&&\end{matrix}\right),$$
where $0\leq c_1\leq \dots \leq c_k$. This is a variant of the
definition of the type $\nu(R)$ of an antisymmetric matrix
$R\in\Mat_{r}(\lri)$ given in~\cite[Section 3.1]{AKOVI/10}. By
definition, we have
$\nu(R)=(a_1,\dots,a_{\lfloor r/2\rfloor})\in(\N_0\cup\{\infty\})^{\lfloor r/2\rfloor}$,
where $0\leq a_1 \leq \dots \leq a_{\lfloor r/2\rfloor}$ if
$$\widetilde{\nu}(R) = \begin{cases}
  (a_1,a_1,\,a_2,a_2,\,\dots,a_{r/2},a_{r/2})&\text{ if $r$ is
    even,}\\ (a_1,a_1,\,a_2,a_2,\,\dots,a_{(r-1)/2},a_{(r-1)/2},\infty)&\text{
    if $r$ is odd.}\end{cases}$$

The definition of $\nu$ takes into account the fact that the
elementary divisors of an antisymmetric matrix of even size come in
pairs.  Analogously to \cite[Lemma~3.4]{AKOVI/10} we have the
following.
\begin{lemma}
Let $w\in\Hom_\lri(\mfg',\lri)^*$ and $N\in\N_0$. Consider the element
$\psi\in\Irr_N(\mfg')$ represented by $w$. Let
$\bfa:=(a_1,\dots,a_{\lfloor r/2 \rfloor})=\nu(\calR(\underline{w}))$ and
$\bfc:=(c_1,\dots,c_k)=\widetilde{\nu}(\calS(\underline{w})\cdot\diag(\pi^{b_1},\dots,\pi^{b_k}))$. Then
\begin{align*}
 | \mfg : \Rad(B_{\psi})| &= q^{2\sum_{i=1}^{\lfloor r/2
     \rfloor}(N-\min\{a_i,N\})}\quad\text{ and}\\
 | \mfg : \mfg_{\psi,2} | &= q^{\sum_{i=1}^k(N-\min\{c_i,N\})}.
\end{align*}
\end{lemma}
Let $N\in\N_0$. Given an antisymmetric matrix
$\overline{R}\in\Mat_{r}(\lri/\mfp^N)$, we set
$\nu(\overline{R}):=(\min\{a_i,N\})_{i\in[\lfloor r/2
    \rfloor]}\in([N]_0)^{\lfloor r/2 \rfloor}$, where $\bfa=\nu(R)$ is
the type of any lift $R$ of $\overline{R}$ under the natural
surjection $\Mat_r(\lri)\rightarrow\Mat_r(\lri/\mfp^N)$. Given
$\overline{S}\in\Mat_{r\times k}(\lri/\mfp^N)$, the vector
$\widetilde{\nu}(\overline{S})\in([N]_0)^k$ is defined similarly. We
set $W_N(\lri) := \left((\lri/\mfp^N)^d\right)^*$.

Given $N\in\N_0$, $\bfa\in\N_0^{\lfloor r/2 \rfloor},
 \bfc\in\N_0^k$, we set
\begin{equation}\label{def:N}
 \calN^\lri_{N,\bfa,\bfc} := \# \left\{\bfy \in W_N(\lri) \vert \;
 \nu(\calR(\bfy))=\bfa, \, \widetilde{\nu}(\calS(\bfy)\cdot
 \diag(\pi^{b_1},\dots,\pi^{b_k}))=\bfc \right\}.
\end{equation}
In analogy with \cite[Proposition~3.1]{AKOVI/10} we have the following.
\begin{proposition}\label{pro:zeta=poincare}
If $p$ is odd or if $p=2$ and $c\geq4$, then
\begin{equation}\label{def:poincare}
 \zeta_{\bfG(\lri)}(s) =
 \sum_{\stackrel{N\in\N_0,}{\bfa\in\N_0^{\lfloor r/2\rfloor},
     \,\bfc\in\N_0^k}}\calN^\lri_{N,\bfa,\bfc}q^{-\sum_{i=1}^{\lfloor
     r/2 \rfloor}(N-a_i)s -
   \sum_{i=1}^k(N-c_i)}=:\calP_{\calR,\calS,\lri}(s).
\end{equation}
\end{proposition}
\begin{corollary}\label{cor:power series in q}
Under the hypothesis of Proposition~\ref{pro:zeta=poincare} the zeta function $\zeta_{\bfG(\lri)}(s)$ is a power
series in $q^{-s}$, that is the degrees of continuous irreducible
complex representations of $\bfG(\lri)$ are powers of $q$.
\end{corollary}
It is interesting to ask whether the conclusion of
Corollary~\ref{cor:power series in q} holds without the hypothesis
of Proposition~\ref{pro:zeta=poincare}.

\subsubsection{$\mfp$-Adic integration}\label{subsubsec:integration}
As explained in \cite[Section~2.2]{Voll/10}, we can express the
Poincar\'e series $\calP_{\calR,\calS,\lri}(s)$ defined
in~\eqref{def:poincare} in terms of a $\mfp$-adic integral. We define
\begin{multline}
\calZ_\lri(\rho,\sigma,\tau) := \\ \int_{(x,\bfy)\in\mfp\times
  W(\lri)}|x|_\mfp^\tau \prod_{j=1}^u\frac{\| F_j(\bfy) \cup
  F_{j-1}(\bfy)x^2\|_\mfp^\rho}{\| F_{j-1}(\bfy)
  \|_\mfp^\rho}\prod_{l=1}^v\frac{\| G_l(\bfy) \cup
  G_{l-1}(\bfy)x\|_\mfp^\sigma}{\| G_{l-1}(\bfy)
  \|_\mfp^\sigma}d\mu(x,\mathbf{y}),\label{equ:integral}
\end{multline}
where $W(\lri)=(\lri^d)^*$, the additive Haar measure $\mu$ on
$\lri^{d+1}$ is normalised so that $\mu(\lri^{d+1}) = 1$, and
\begin{align*}
  2u & = \max \{ \rk_{\textrm{Frac}(\lri)}(\mathcal{R}(\bfz)) \mid \bfz \in
  \lri^{d} \}, \nonumber \\
  v &= \max \{ \rk_{ \textrm{Frac}(\lri)}(\mathcal{S}(\bfz)) \mid \bfz \in
  \lri^{d} \}, \nonumber \\
  F_j(\bfY) & = \{ f \mid f=f(\bfY)
  \text{ a principal $2j \times 2j$ minor of $\mathcal{R}(\bfY)$} \},
  \nonumber \\
  G_l(\bfY) &= \left\{ g \mid g=g(\bfY) \text{ an $l\times l$ minor of
  }\calS(\bfY)\cdot\diag(\pi^{b_1},\dots,\pi^{b_k})\right\},\\ \lVert
  H(X,\bfY) \rVert_\mfp & = \max \left\{ \lvert h(X,\bfY) \rvert_\mfp \mid
  h \in H \right\} \text{ for a finite set }H\subset\lri[X,\bfY].\nonumber
\end{align*}
As in \cite[Section~2.2]{Voll/10} one shows that
\[
\calP_{\calR,\calS,\lri}(s) = 1 +
(1-q^{-1})^{-1}\calZ_\lri(-s/2,-1,u s+v-d-1).
\]
This yields the following corollary to
Proposition~\ref{pro:zeta=poincare}.
\begin{corollary}\label{cor:zeta=integral}
Under the hypothesis of Proposition~\ref{pro:zeta=poincare} we have
$$\zeta_{\bfG(\lri)}(s) = 1 + (1-q^{-1})^{-1}\calZ_\lri(-s/2,-1,u
s+v-d-1).$$ In particular, the zeta function $\zeta_{\bfG(\lri)}(s)$
is a rational function in $q^{-s}$.
\end{corollary}
Indeed, the rationality in $q^{-s}$ of integrals
like~\eqref{equ:integral} is a well known fact in the theory of
$\mfp$-adic integration; cf., for instance,~\cite{Denef/91}.

\subsection{Proof of Theorem~\ref{thmABC:thm A}}\label{subsec:proof thm A}

We now return to the global setup of Theorem~\ref{thmABC:thm
  A}. Recall that $\bfG=\bfG_\Lambda$ is a unipotent group scheme,
  defined by a nilpotent $\Gri$-Lie lattice $\Lambda$ of nilpotency
  class $c$, where $\Gri=\Gri_K$ is the ring of integers of a number
  field $K$. For a finite extension $L$ of $K$, with ring of integers
  $\Gri_L$, we wish to describe the zeta function of the $\T$-group
  $\bfG(\Gri_L)$.  By Proposition \ref{pro:Fine Euler prod} we have
\begin{equation}\label{equ:euler extension}
 \zeta_{\bfG(\Gri_L)}(s) = \prod_{\mfP}\zeta_{\bfG(\Gri_{L,\mfP})}(s),
\end{equation}
where the product ranges over the non-zero prime ideals of~$\Gri_L$.
For such a prime ideal $\mfP$ of $\Gri_L$, dividing the prime ideal
$\mfp$ of $\Gri$, we write $\Lri$ for the local ring $\Gri_{L,\mfP}$
and $\lri$ for~$\Gri_{K,\mfp}$. Further we write $f=f(\Lri,\lri)$ for
the relative degree of inertia. We continue to write $q$ for the
residue field cardinality of $\lri$, and $p$ for its residue field
characteristic, so that $\vert \Lri/\mfP \vert =q^f$.

Assume that $p$ is odd or that $p=2$ and $c\geq4$. The $\mfp$-adic formalism developed in
Sections~\ref{subsec:kirillov} is applicable to the factor
$\zeta_{\bfG(\Lri)}(s)$ in~\eqref{equ:euler extension}. Two facts are
key to proving Theorem~\ref{thmABC:thm A}: Firstly, we observe that
the polynomials occurring in the integrand of the $\mfp$-adic
integral~\eqref{equ:integral} are defined over~$\Gri$, so that
effectively only the domain of integration depends on
$\Lri$. Secondly, we exploit that there is, as we shall explain, a
uniform formula for~\eqref{equ:integral} in which only the residue
field of $\Lri$ enters.

A priori, the $\Lri$-bases ${\bf e}$ and ${\bf f}$ defined in
Section~\ref{subsubsec:poincare} -- and thus the matrices
$\calR(\bfY)$ and $\calS(\bfY)$ and the data $\bfb=(b_1,\dots,b_d)$ --
are defined only locally. As we do not assume that $\Gri$ or $\Gri_L$
are principal ideal domains, we may not hope for a global analogue of
the construction of the bases ${\bf e}$ and ${\bf f}$. Instead, we
choose an $\Gri$-basis ${\bf f}=(e_{r-k+1},\dots,e_{r-k+d})$ for
a free finite index $\Gri$-submodule of the $\Gri$-isolator
$\iota(\Lambda'(\Gri))$, and extend it to an $\Gri$-basis ${\bf e}$
for a free finite index $\Gri$-submodule of $\Lambda(\Gri)$;
cf.\ Lemma~\ref{lem:Free centre}. Provided $p$ does not divide the
index of the latter in $\Lambda(\Gri)$, we may use this basis ${\bf
  e}$ to obtain an $\Lri$-basis for $\Lambda(\Lri)$ in the analysis of
Section~\ref{subsec:kirillov}, with $\bfb=(0,\dots,0)$. This ensures
that the polynomials occurring in \eqref{equ:integral} are defined
over $\Gri$. Theorem~\ref{thmABC:thm A} now follows formally by the
arguments given in~\cite[Section~4]{AKOVI/10}, as the integral
\eqref{equ:integral} can be expressed in terms of an integral of the
form \cite[Equation (4.1)]{AKOVI/10}.

\subsection{Nilpotency class $2$}\label{subsec:class 2}

In nilpotency class $2$, many of the constructions given in the
previous sections can be made more directly, allowing us to deduce
slightly stronger results. While some of these modifications will be
known to the experts, we record them here for completeness. Throughout
this section, let $\Lambda$ be a class-$2$-nilpotent $\Gri$-Lie
lattice, with $\rk_\Gri(\Lambda)=h$, say.  Note that we do not assume
that $\Lambda'\subseteq 2\Lambda$ here.
\subsubsection{}\label{subsubsec:group schemes class 2}
We start by constructing a group scheme $\bfG_\Lambda$ associated
to~$\Lambda$, which coincides with the group scheme defined in
Section~\ref{subsubsec:group schemes} if~$\Lambda'\subseteq
2\Lambda$. We fix an $\Gri$-basis $(x_1,\dots,x_h)$ for $\Lambda$. Let
$R$ be an $\Gri$-algebra and consider $\Lambda(R)=\Lambda\otimes_\Gri
R$. We also write $(x_{1},\dots,x_h)$ for the $R$-basis
$(x_{1}\otimes1,\dots,x_h\otimes1)$ for~$\Lambda(R)$.

Any element $g\in\Lambda(R)$ can be expressed uniquely as
$g=\sum_{i=1}^{h}a_{i}x_{i}$, for some $\bfa=(a_{1},\dots,a_{h})\in
R^{h}$. Adopting multiplicative notation, we identify $g$ with the
formal monomial $\bfx^{\bfa}=x_{1}^{a_{1}}\cdots x_h^{a_{h}}$, and
define a group multiplication $*$ on the set of all such monomials by
defining group commutators via the Lie bracket on~$\Lambda(R)$. More
precisely, we define, for $1\leq i<j \leq r$ and $a_i,a_j\in R$,
$$x_i^{a_i}*x_j^{a_j} = x_i^{a_i}x_j^{a_j}, \quad
x_j^{a_j}*x_i^{a_i}=x_i^{a_i}x_j^{a_j}\bfx^{a_ia_j\boldsymbol\lambda_{ij}},$$
where
$\boldsymbol\lambda_{ij}=(\lambda_{ij}^1,\dots,\lambda_{ij}^{h})\in
\Gri^{h}$ is defined via the identity
$[x_i,x_j]=\sum_{k=1}^{h}\lambda_{ij}^kx_k$ in $\Lambda(R)$.
Extending this to the set of all monomials in the obvious way, we
obtain polynomials
$M_{i}(X_{1},\dots,X_h,\tilde{X}_{1},\dots,\tilde{X}_h)$, for
$i=1,\dots,h$, over $\Gri$ such that
\[
\bfx^{\bfa}*\bfx^{\bfa'}=\bfx^{\bfa+\bfa'+(M_{i}(\bfa,\bfa'))_{i}}.\]
Similarly, there are polynomials $I_{i}(X_{1},\dots,X_h)$, for
$i=1,\dots,h$, over $\Gri$ such that \[
(\bfx^{\bfa})^{-1}=\bfx^{-\bfa+(I_{i}(\bfa))_{i}}.\]
This defines a unipotent group scheme $\bfG_{\Lambda}$ over~$\Gri$,
isomorphic as a scheme to affine $h$-space over~$\Gri$, representing
the group functor\[ R\longmapsto(\{\bfx^{\bfa} \mid
\bfa\in R^h\},*).\] Note that when $c=2$ and
$\Lambda'\subseteq2\Lambda$, the group scheme $\mathbf{G}_{\Lambda}$
is isomorphic to the one constructed in Section~\ref{subsubsec:group
schemes} by means of the Hausdorff series.

\begin{remark}\label{rem:class 2}
  The construction sketched here mirrors, of course, the classical
  construction of a unipotent group scheme $\bfG$ over $\Z$ associated
  to a $\T$-group $G$ of nilpotency class~$2$ by means of a
  Mal'cev-basis. This allows one to recover $G$ as $\bfG(\Z)$ and to
  define the group~$G^R:=\bfG(R)$, for any $\Z$-module $R$;
  c.f.~\cite[Sections 1 and 5]{GSegal/84}. The point of view taken in
  the present paper allows us to study group schemes defined over
  proper extensions of~$\Z$.
\end{remark}

\subsubsection{}
The main results in Section~\ref{subsec:kirillov} all assume that $p$
is odd or that $p=2$ and $c\geq4$. In the current section we formulate
a Kirillov orbit formalism for group schemes of nilpotency class $2$
which is valid for all primes $p$, without restriction. We indicate
how the analysis of Sections~\ref{subsubsec:poincare}
and~\ref{subsubsec:integration} simplifies in this case. Note that the
proof of Theorem~\ref{thmABC:thm A} may still require the exclusion of
finitely many places.

Let $(x_1,\dots,x_{h})$ be an $\Gri$-basis for $\Lambda$, and let
$\bfG=\bfG_\Lambda$ be the group scheme defined in
Section~\ref{subsubsec:group schemes class 2}. For any $\Gri$-algebra
$R$, we have $\bfG(R)=\{\bfx^\bfa\mid \bfa\in R^h\}$, and an obvious
bijection
\[
 \lambda_{R}:\bfG(R)\longrightarrow\Lambda(R),\qquad\bfx^{\bfa}\longmapsto\sum_{i=1}^{h}a_{i}x_{i}.\]

We now assume that $R=\lri$, a compact discrete valuation ring of
characteristic zero, and write $\lambda$ for $\lambda_{\lri}$.  We
further write $\mfg$ for the $\lri$-Lie lattice $\Lambda(\lri)$, and
$\widehat{\mfg}$ for its Pontryagin
dual~$\Hom_{\Z}^{\textup{cts}}(\mfg,\C^\times)$.  For any
$\psi\in\widehat{\mfg}$ we consider the form \[
B_{\psi}:\mfg\times\mfg\longrightarrow\C^\times,\quad(x,y)\longmapsto\psi([x,y])\]
with radical $\mathrm{Rad}(B_{\psi}):=\{x\in\mfg\mid\forall
y\in\mfg:B_{\psi}(x,y)=1\}$.  For $g\in\bfG(\lri)$ and $x\in\mfg$, we
define the co-adjoint action
$\bfG(\lri)\times\widehat{\mfg}\mapsto\widehat{\mfg}$ by
$(g,\psi)\mapsto\ad^{*}(g)\psi$, where $\ad^{*}(g)\psi$ is the map
given by\[ x\longmapsto\psi(x+[\lambda(g),x]).\] Let $\Stab(\psi)$
denote the stabiliser in $\bfG(\lri)$ of $\psi$ under the co-adjoint
action, and write $\Omega(\psi)$ for the co-adjoint orbit of~$\psi$.

\begin{lemma}
 For any $\psi\in\widehat{\mfg}$, the group $\Stab(\psi)$ contains
 $Z(\bfG(\lri))$, and
$$\Stab(\psi) = \lambda^{-1}(\Rad(B_\psi)).$$ In particular,
 \begin{equation}
 |\Omega(\psi)|=|\bfG(\lri):\Stab(\psi)|=|\mfg:\mathrm{Rad}(B_{\psi})|<\infty.\label{equ:stab
   index}\end{equation}
 \end{lemma}
\begin{proof}
We have
\begin{align*}
 \Stab(\psi) & =\{g\in\bfG(\lri)\mid\forall
 x\in\mfg:\psi(x+[\lambda(g),x])=\psi(x)\,\}\\ &
 =\{g\in\bfG(\lri)\mid\forall x\in\mfg:\psi([\lambda(g),x])=1\,\}\\ &
 = \lambda^{-1}(\Rad(B_\psi)).
\end{align*}
This proves the first two assertions.  The claims regarding the orbit
size follow from the orbit stabiliser theorem and the fact that $\psi$
is continuous.
\end{proof}

A key role in the Kirillov orbit method is played by polarising
subalgebras for characters $\psi\in\widehat{\mfg}$, that is subalgebras
$P$ of $\mfg$ with the property that $B_{\psi}|_{P\times P}=1$, and
which are maximal with respect to this property.

\begin{lemma}\label{lem:polar}
 Let $\psi\in\widehat{\mfg}$. Then there exists a polarising
subalgebra for $\psi$.
\end{lemma}

\begin{proof}
  There exists $n\in\N$ and a homomorphism
  $\overline{\psi}:\Lambda(\lri/\mfp^{n})\rightarrow\mathbb{C}^{\times}$
  such that $\psi$ factors through $\overline{\psi}$ via the natural
  surjection $p_{n}:\mfg\rightarrow\Lambda(\lri/\mfp^{n})$, that is
  $\psi=\overline{\psi}\circ p_{n}$. Thus the form $B_{\psi}$ factors
  through the form \[
  B_{\overline{\psi}}:\Lambda(\lri/\mfp^{n})\times\Lambda(\lri/\mfp^{n})\longrightarrow\mathbb{C}^{\times},\quad(x,y)\longmapsto\overline{\psi}([x,y]).\]
  Let $P_{\overline{\psi}}$ be a polarising subalgebra for
  $\overline{\psi}$ in $\Lambda(\lri/\mfp^{n})$; see, for instance,
  \cite[Lemma~4]{Howe-nilpotent/77}.  Clearly
  $P_{\psi}:=p_{n}^{-1}(P_{\overline{\psi}})$ has the desired
  properties.
\end{proof}

\begin{lemma}\label{lem:linear}
  Let $P\subseteq\mfg$ be an $\lri$-subalgebra.  Let
  $\psi\in\widehat{\mfg}$ be such that~$\psi(P')=1$. Then
  $H:=\lambda^{-1}(P)$ is a subgroup of $\bfG(\lri)$ and the
  restriction $\psi\circ\lambda\vert_{H}$ is a $1$-dimensional
  representation of~$H$.
\end{lemma}
\begin{proof}
 Let $h_1=\bfx^{\bfa_1}$ and $h_2={\bfx}^{\bfa_2}\in H$.  It is easy to
 check that
 \begin{equation*}
  h_1h_2=\bfx^{\bfa_1}\,\bfx^{\bfa_2}=\bfx^{\bfa_1+\bfa_2+\mathbf{u}},
 \end{equation*} for some $\mathbf{u}\in \lri^{h}$ such that $\bfx^\mathbf{u}\in H'$. Clearly $\mathbf{u}\in P'$, and so
 \begin{equation}
   \lambda(h_1h_2)=\bfa_1 + \bfa_2 + \mathbf{u} \in
   P,\label{equ:product}
 \end{equation}
 and hence $h_1h_2\in H$.  Similarly, given $h=\bfx^{\bfa}\in H$, we
 have \[ h^{-1}=(\bfx^{\bfa})^{-1}=\bfx^{-\bfa+\mathbf{v}},\] for some
 $\mathbf{v}\in \lri^{h}$ such that $\bfy^\mathbf{v}\in H'$. Then
 $\mathbf{v}\in P'$, and so\[ \lambda(h^{-1})=-\bfa+\mathbf{v}\in P.\]
 Thus $h^{-1}\in H$, and so $H$ is a subgroup of
 $\bfG(\lri)$. Using~\eqref{equ:product} and the fact that
 $\psi(P')=1$, we obtain
 \begin{equation*}
  \psi(\lambda(h_1h_2))=\psi(\bfa_1+\bfa_2+{\bf u})=\psi(\bfa_1)\psi(\bfa_2)=\psi(\lambda(h_1))\psi(\lambda(h_2)),
 \end{equation*}
 so $\psi\circ\lambda\vert_{H}$ is a homomorphism of $H$, as asserted.
\end{proof}

For every $\psi\in\widehat{\mfg}$, we choose a polarising subalgebra
$P_{\psi}$ for $\psi$ in $\mfg$;
cf. Lemma~\ref{lem:polar}. Lemma~\ref{lem:linear} asserts that
$H_{\psi}:=\lambda^{-1}(P_{\psi})$ is a subgroup of $\bfG(\lri)$, and
that $\psi\circ\lambda\vert_{H_{\psi}}$ is a $1$-dimensional
representation of $H_\psi$. We define the representation
\[
\pi(\psi)=\Ind_{H_{\psi}}^{\bfG(\lri)}(\psi\circ\lambda\vert_{H_{\psi}}).
\]
Recall that $\Irr(\bfG(\lri))$ denotes the set of isomorphism classes
of continuous, irreducible complex representations of $\bfG(\lri)$.
The following result establishes a Kirillov orbit method for the group $\bfG(\lri)$.

\begin{proposition}\label{pro:Kirillov-class2}
 For every $\psi\in\widehat{\mfg}$ the representation $\pi(\psi)$ is
 irreducible, of dimension $\vert\Omega(\psi)\vert^{1/2}$, and all
 representations in $\Irr(\bfG(\lri))$ arise in this way.  Given
 $\phi,\psi\in\widehat{\mfg}$, the representations $\pi(\phi)$ and
 $\pi(\psi)$ are isomorphic if and only if
 $\Omega(\phi)=\Omega(\psi)$.  They are twist-equivalent if and only
 if $\phi|_{\mfg'}=\psi|_{\mfg'}$.
\end{proposition}
\begin{proof}
 The character $\chi_{\pi(\psi)}$ of the representation $\pi(\psi)$ is
 given by\[
 \chi_{\pi(\psi)}(g)=|\Omega(\psi)|^{-1/2}\sum_{\omega\in\Omega(\psi)}\omega(\lambda(g)),\quad\text{for
 }g\in\bfG(\lri);\] see, for instance, \cite[Proposition
 1]{Kazhdan/77}, with $\log$ replaced by $\lambda$. This character
 formula immediately shows that the isomorphism class of $\pi(\psi)$
 only depends on $\Omega(\psi)$, that $\pi(\psi)$ is irreducible, and
 that $\langle\pi(\phi),\pi(\psi)\rangle=0$ if
 $\Omega(\phi)\neq\Omega(\psi)$. With \eqref{equ:stab index} it also
 implies that
 $\dim\pi(\psi)=|\Omega(\psi)|^{1/2}=|\mfg:\mathrm{Rad}(B_{\psi})|^{1/2}$.
 For $n\in\N_{0}$, the set of continuous additive characters of $\mfg$
 which factor through $\Lambda(\lri/\mfp^{n})$ is a union of
 co-adjoint orbits. The representations of $\bfG(\lri)$ associated to
 these orbits all factor through the finite group
 $\bfG(\lri/\mfp^{n})$, of order
 $q^{nh}=|\Lambda(\lri/\mfp^{n})|$. But $q^{nh}$ is also the sum of
 the squares of the dimensions of the irreducible representations of
 $\bfG(\lri/\mfp^{n})$, so every irreducible representation of
 $\bfG(\lri/\mfp^{n})$ must be of the form $\pi(\psi)$, for some
 $\psi$.

Finally, let $\chi$ be a continuous $1$-dimensional representation of
$\bfG(\lri)$. We have \begin{equation}
  \chi\otimes\pi(\psi)=\Ind_{H_{\psi}}^{\bfG(\lri)}\left(\chi|_{H_{\psi}}\otimes(\psi\circ\lambda|_{H_{\psi}})\right).\end{equation}
Since $\bfG(\lri)'\leq H_{\psi}$, this implies that two
representations $\pi(\phi)$ and $\pi(\psi)$ are twist-equivalent if
and only if $\phi|_{\mfg'}=\psi|_{\mfg'}$.
\end{proof}
We record the following immediate consequence of
Proposition~\ref{pro:Kirillov-class2}.
\begin{corollary}\label{cor:zeta kirillov}
 We have
$$\zeta_{\bfG(\lri)}(s)=\sum_{\psi\in\widehat{\mfg'}}\vert\mfg:\Rad(B_{\psi})\vert^{-s/2}.$$
\end{corollary}
We note that, in contrast to Theorem~\ref{thm:zeta radical}, the
formula given in Corollary~\ref{cor:zeta kirillov} is valid for all
primes~$p$, and is somewhat simpler than~\eqref{equ:zeta}. The
formalism developed in Sections~\ref{subsubsec:kirillov} and
\ref{subsubsec:integration} applies, without the assumption that
$p>2$, and simplifies as follows. We have $k=0$, as $\mfg'\leq \mfz$,
so the matrix $\calS$ makes no appearance. In analogy with the number
$\calN^\lri_{N,\bfa,\bfc}$ defined in~\eqref{def:N}, we therefore set,
for $N\in\N_0$, $\bfa\in\N_0^{\lfloor r/2 \rfloor}$,
\[
 \calN^\lri_{N,\bfa} := \# \left\{\bfy \in W_N(\lri) \mid
 \nu(\calR(\bfy))=\bfa \right\}.
\]
The class-$2$-analogue of Proposition~\ref{pro:zeta=poincare} is the
following.
\begin{proposition}\label{pro:zeta=poincare class 2}
 We have
 \begin{equation}\label{equ:zeta=poincare class 2}
 \zeta_{\bfG(\lri)}(s) = \sum_{N\in\N_0, \bfa\in\N_0^{\lfloor r/2
     \rfloor}}\calN^\lri_{N,\bfa}q^{-\sum_{i=1}^{\lfloor r/2\rfloor}
   (N-a_i)s}=:\calP_{\calR,\lri}(s).
\end{equation}
\end{proposition}
The Poincar\'e series $\calP_{\calR,\lri}(s)$ may be expressed in
terms of the $\mfp$-adic integral~\eqref{equ:integral}, simplified by
the fact that $v=0$. The following is analogous to
Corollary~\ref{cor:zeta=integral}.

\begin{corollary}\label{cor:power series in q class 2}
 The zeta function $\zeta_{\bfG(\lri)}(s)$ is a rational function in
 $q^{-s}$.
\end{corollary}

In the next section we compute the Poincar\'e series
$\calP_{\calR,\lri}(s)$ directly for groups of type $F$, $G$ and~$H$,
bypassing the need to evaluate $\mfp$-adic integrals
like~\eqref{equ:integral}.

\section{Proof of Theorem~\ref{thmABC:thm C}}\label{sec:proofThmC}

Recall that $n\in\N$ and $\delta\in\{0,1\}$. Let
$\Lambda\in\{\calF_{n,\delta},\calG_{n},\calH_{n}\}$ be one of the Lie
rings defined in Definition~\ref{def:lattices}, and
$\bfG=\bfG_\Lambda\in\{F_{n,\delta}, G_{n}, H_{n}\}$ the associated
group scheme. As before, given a number field $K$ with ring of
integers~$\Gri$, and a non-zero prime ideal $\mfp$ of $\Gri$, we write
$\lri=\Gri_\mfp$ for the completion of $\Gri$ at $\mfp$, and $q$ for
the residue field cardinality $\vert \lri/\mfp\vert$. We write $\mfg$
for $\Lambda(\lri)$ and set $r=\rk_\lri(\mfg/\mfg')$,
$d=\rk_\lri(\mfg')$, in accordance with the notation introduced in
Section~\ref{subsubsec:poincare}. Note that, in all three cases, we
have~$n=\lfloor r/2 \rfloor$. We are looking to compute the zeta
function $\zeta_{\bfG(\lri)}(s)$. By
Proposition~\ref{pro:zeta=poincare class 2} it suffices to study the
Poincar\'e series $\calP_{\calR,\lri}(s)$ associated to the commutator
matrix $\calR(\bfY)=\calR_{\Lambda}(\bfY)$ of the relevant Lie lattice
with respect to the $\Z$-bases given in the presentations in
Definition~\ref{def:lattices}. These are as follows:
\begin{itemize}
 \item $\calR_{\calF_{n,\delta}}(\bfY)$ is the generic antisymmetric
   $(2n+\delta)\times (2n+\delta)$-matrix in the variables $Y_{ij}$,
   $1\leq i<j\leq 2n+\delta$. We have $r=2n+\delta$ and
   $d=\binom{2n+\delta}{2}$.
 \item $\calR_{\calG_{n}}(\bfY)=\left(\begin{matrix} &
   \textrm{M}(Y_{ij})\\ -\textrm{M}(Y_{ij})^{\mathrm{t}}\end{matrix}\right)$,
   where $\textrm{M}(Y_{ij})$ is the generic $n\times n$-matrix in the
   variables $Y_{ij}$, $1\leq i,j\leq n$. We have $r=2n$ and $d=n^2$.
 \item $\calR_{\calH_{n}}(\bfY)=\left(\begin{matrix} &
   \textrm{S}(Y_{ij})\\ -\textrm{S}(Y_{ij})\end{matrix}\right)$, where
   $\textrm{S}(Y_{ij})$ is the generic symmetric $n\times n$-matrix in
   the variables $Y_{ij}$, $1\leq i\leq j\leq n$. We have $r=2n$ and
   $d=\binom{n+1}{2}$.
\end{itemize}

It is advantageous to re-organise the respective series
$\calP_{\calR,\lri}(s)$ in the following way. Let
$I=\{i_{1},\dots,i_{l}\}_{<}\subseteq[n-1]_{0}$, and recall that, for
$j\in[n]_{0}$, we write $\mu_{j} = i_{j+1}-i_{j}$, with $i_0=0$
and~$i_{l+1}=n$. For $\bfr_{I}=(r_{i_{1}},\dots,r_{i_{l}})\in\N^{I}$,
we set $N=\sum_{i\in I}r_{i}$,
$W_N(\lri)=\left((\lri/\mfp^N)^d\right)^*$ and define
\begin{multline*}
\rmN^\lri_{I,\bfr_{I}}(\bfG):=\{\underline{w}\in
W_{N}(\lri)\mid\nu(\calR_\Lambda(\underline{w}))=\\(\underbrace{0,\dots,0}_{\mu_{l}},\underbrace{r_{i_{l}},\dots,r_{i_{l}}}_{\mu_{l-1}},\underbrace{r_{i_{l}}+r_{i_{l-1}},\dots,r_{i_{l}}+r_{i_{l-1}}}_{\mu_{l-2}},\dots,\underbrace{N,\dots,N}_{\mu_{0}})\in\N_0^{\lfloor
r/2 \rfloor}\}.
\end{multline*}
Note that $W_N(\lri)$ is partitioned by such sets. In particular, for
every $\underline{w}\in W_N(\lri)$ the type
$\nu(\calR_{\Lambda}(\underline{w}))$ always contains at least one
zero, as $\calR_{\Lambda}(\underline{w})\not\equiv 0 \bmod \mfp$.

We can now rewrite~\eqref{equ:zeta=poincare class 2} as
\begin{equation}\label{rewrite zeta=poincare}
 \zeta_{\bfG(\lri)}(s)=\calP_{\calR,\lri}(s) =
 \sum_{I\subseteq[n-1]_{0}}\sum_{\bfr_{I}\in\N^{I}}|\rmN^\lri_{I,\bfr_{I}}(\bfG)|
 q^{-s\sum_{i\in I}r_{i}(n-i)}.
\end{equation}

We prove Theorem~\ref{thmABC:thm C} by computing the quantities
$|\rmN^\lri_{I,\bfr_{I}}(\bfG)|$ explicitly in the three cases; see
Proposition~\ref{pro: Mainthm-formulae}. We start with a few
preliminary definitions and a lemma. Given $j\in\N$ and a ring~$R$, we
denote by $\Alt_j(R)$ and $\Sym_j(R)$ the antisymmetric and symmetric
matrices in $\Mat_j(R)$, respectively. Let $i\in[n]_{0}$, and define
\begin{align*}
 \Alt_{2n+\delta,2(n-i)}(\Fq) &= \{x\in\Alt_{2n+\delta}(\Fq)\mid
 \rk(x)=2(n-i)\},\\ \Mat_{n,n-i}(\Fq) &= \{x\in\Mat_{n}(\Fq)\mid
 \rk(x)=n-i\},\\ \Sym_{n,n-i}(\Fq) &= \{x\in\Sym_{n}(\Fq)\mid \rk(x)=n-i\}.
\end{align*}
\begin{lemma}\label{lem:ranks}
For $i\in[n]_0$ we have
\begin{align}
 \vert\Alt_{2n+\delta,2(n-i)}(\Fq)\vert&=
 \binom{n}{i}_{q^{-2}}(q^{-2(i+\delta)-1};q^{-2})_{n-i}\cdot q^{\binom{2n+\delta}{2}
   - \binom{2i+\delta}{2}}, \label{ranks
   F}\\ \vert\Mat_{n,n-i}(\Fq)\vert&=
 \binom{n}{i}_{q^{-1}}(q^{-i-1};q^{-1})_{n-i}\cdot q^{n^2-i^2},\label{ranks
   G}\\ \vert\Sym_{n,n-i}(\Fq)\vert&=
 (q^{-2};q^{-2})^{-1}_{\lfloor(n-i)/2\rfloor}(q^{-i-1};q^{-1})_{n-i}\cdot q^{\binom{n+1}{2}-\binom{i+1}{2}}.\label{ranks
   H}
\end{align}
\end{lemma}
\begin{proof}
These are all well-known; see, for instance, \cite[Section~7]{CarlitzHodges/56}
for \eqref{ranks F}, \cite[Proposition~3.1]{LaksovThorup/94} for
\eqref{ranks G} and \cite[Lemma~10.3.1]{Igusa/00} for \eqref{ranks H}.
\end{proof}

\begin{definition}\label{def:polys thm C}
For $I=\{i_1,\dots,i_l\}_<\subseteq[n-1]_0$ set
\begin{align}
 f_{F_{n,\delta},I}(X) &
 =\binom{n}{I}_{X^2}(X^{2(i_{1}+\delta)+1};X^2)_{n-i_{1}},\label{eq:f
 type F}\\ f_{G_{n},I}(X) &
 =\binom{n}{I}_{X}(X^{i_1+1};X)_{n-i_1},\label{eq:f type G}\\
 f_{H_{n},I}(X) &
 =\left(\prod_{j=1}^{l}(X^2;X^2)^{-1}_{\lfloor\mu_{j}/2\rfloor}\right)(X^{i_1+1};X)_{n-i_1}.\nonumber
 \end{align}
\end{definition}

\begin{remark}
 We note that comparison with Lemma~\ref{lem:ranks} shows that the
 polynomials $f_{\bfG,\{i\}}(X)$, where
 $\bfG\in\{F_{n,\delta},G_n,H_n\}$ and $i\in[n-1]_0$, give, in effect,
 the Poincar\'e polynomials of the determinantal varieties of
 (symmetric or antisymmetric) matrices of given rank. In
 Proposition~\ref{pro:combinatorial formulae} we give an
 interpretation of the polynomials $f_{F_{n,\delta},I}(X)$ and
 $f_{G_n,I}(X)$ in terms of generating functions over descent classes
 in Weyl groups of type~$B$. In Conjecture~\ref{con:L} we record a
 conjectural formula of this type for the polynomials~$f_{H_n,I}(X)$.
\end{remark}

\begin{proposition}\label{pro: Mainthm-formulae}
 Let $I\subseteq[n-1]_{0}$ and $\bfr_{I}\in\N^{I}$.  Then we have
\begin{align}
    |\rmN^\lri_{I,\bfr_{I}}(F_{n,\delta})| &=
    f_{F_{n,\delta},I}(q^{-1}) \, q^{\sum_{i\in
        I}r_{i}\left(\binom{2n+\delta}{2}-\binom{2i+\delta}{2}\right)}, \label{FormulaF}\\
  \label{FormulaG}
    |\rmN^\lri_{I,\bfr_{I}}(G_{n})| &= f_{G_n,I}(q^{-1}) \,
    q^{\sum_{i\in I}r_{i}(n^{2}-i^{2})},\\
  \label{FormulaH}
    |\rmN^\lri_{I,\bfr_{I}}(H_{n})| &=
    f_{H_n,I}(q^{-1})
    \, q^{\sum_{i\in
        I}r_{i}\left(\binom{n+1}{2}-\binom{i+1}{2}\right)}.
\end{align}
\end{proposition}
\begin{proof}
  We write $\rho_{i_l}:\mathrm{N}^\lri_{I,r_{I}}(\bfG)\rightarrow
  W_{r_{i_l}}(\lri)$ for the map given by reduction of entries modulo
  $q^{r_{i_l}}$. We first prove \eqref{FormulaF}. There are
  $|\Alt_{2n+\delta,2(n-i_{l})}(\Fq)|
  q^{\left(r_{i_{l}}-1\right)\left(\binom{2n+\delta}{2}-\binom{2i_{l}+\delta}{2}\right)}$
  elements in
  $\rho_{{i_{l}}}(\mathrm{N}^\lri_{I,\bfr_{I}}(F_{n,\delta}))$. Each
  such element has
\[
 q^{\left(\sum_{i\in
     I,i<i_{l}}r_{i}\right)\left(\binom{2n+\delta}{2}-\binom{2i_{l}+\delta}{2}\right)}|\rmN^\lri_{I\setminus\{i_{l}\},r_{I\setminus\{i_{l}\}}}(F_{i_{l},\delta})|
\]
 lifts to an element in $\mathrm{N}^\lri_{I,\bfr_{I}}(F_{n,\delta})$. By
 \eqref{ranks F} we thus get
\begin{align*}
 |\rmN^\lri_{I,\bfr_{I}}(F_{n,\delta})| &
 =|\Alt_{2n+\delta,2(n-i_{l})}(\Fq)|q^{\left(N-1\right)\left(\binom{2n+\delta}{2}-\binom{2i_{l}+\delta}{2}\right)}|\rmN^\lri_{I\setminus\{i_{l}\},r_{I\setminus\{i_{l}\}}}(F_{i_{l},\delta})|\\ &
 =\binom{n}{i_l}_{q^{-2}}(q^{-2(i_l+\delta)-1};q^{-2})_{n-i_l}
 q^{N\left(\binom{2n+\delta}{2}-\binom{2i_{l}+\delta}{2}\right)}|\rmN^\lri_{I\setminus\{i_{l}\},r_{I\setminus\{i_{l}\}}}(F_{i_{l},\delta})|.
\end{align*}
Working recursively in this way, we obtain
\begin{align*}
 |\rmN^\lri_{I,\bfr_{I}}(F_{n,\delta})| &
 =\prod_{j=1}^l \binom{i_{j+1}}{i_{j}}_{q^{-2}}(q^{-2(i_j+\delta)-1};q^{-2})_{i_{j+1}-i_{j}}\cdot q^{\left(\sum_{i\in I, i\leq i_j}r_i\right)\left(\binom{2i_{j+1}+\delta}{2} - \binom{2i_j+\delta}{2}\right)}\\
&=\binom{n}{I}_{q^{-2}}(q^{-2(i_1+\delta)-1};q^{-2})_{n-i_1}\cdot q^{\sum_{i\in I}r_i\left(\binom{2n+\delta}{2}-\binom{2i+\delta}{2}\right)}.
\end{align*}

Next we prove \eqref{FormulaG}. There are
$|\Mat_{n,n-i_{l}}(\Fq)|q^{(r_{i_{l}}-1)(n^{2}-i_{l}^{2})}$ elements
in the set $\rho_{{i_{l}}}(\mathrm{N}^\lri_{I,\bfr_{I}}(G_{n}))$. Each
such element has
\[ q^{\left(\sum_{i\in
I,i<i_{l}}r_{i}\right)\left(n^{2}-i_{l}^{2}\right)}|\rmN^\lri_{I\setminus\{i_{l}\},r_{I\setminus\{i_{l}\}}}(G_{i_{l}})|\]
lifts to an element in $\mathrm{N}^\lri_{I,\bfr_{I}}(G_{n})$. By \eqref{ranks
G} we thus get
\begin{align*}
 |\rmN^\lri_{I,\bfr_{I}}(G_{n})| &
 =|\Mat_{n,n-i_{l}}(\Fq)|q^{\left(N-1\right)(n^{2}-i_{l}^{2})}|\rmN^\lri_{I\setminus\{i_{l}\},r_{I\setminus\{i_{l}\}}}(G_{i_{l}})|\\ &
 =\binom{n}{i_l}_{q^{-1}}(q^{-i_l-1};q^{-1})_{n-i_l}q^{N(n^{2}-i_{l}^{2})}|\rmN^\lri_{I\setminus\{i_{l}\},r_{I\setminus\{i_{l}\}}}(G_{i_{l}})|.\end{align*}
Working recursively in this way, we obtain
\begin{align*}
  |\rmN^\lri_{I,\bfr_{I}}(G_{n})| &
  =\prod_{j=1}^{l}\binom{i_{j+1}}{i_{j}}_{q^{-1}}(q^{-i_j-1};q^{-1})_{i_{j+1}-i_j}\cdot q^{\left(\sum_{i\in
      I,i\leq i_{j}}r_{i}\right)(i_{j+1}^{2}-i_{j}^{2})}\\ &
  =\binom{n}{I}_{q^{-1}}(q^{-i_1-1};q^{-1})_{n-i_1}\cdot q^{\sum_{i\in I}r_i(n^2-i^2)}.\end{align*}

Finally we prove \eqref{FormulaH}. There are $|\Sym_{n,n-i_{l}}(\Fq)|
q^{(r_{i_{l}}-1)\left(\binom{n+1}{2}-\binom{i_{l}+1}{2}\right)}$
elements in
$\rho_{{i_{l}}}(\mathrm{N}^\lri_{I,\bfr_{I}}(H_{n}))$. Each such
element has \[ q^{\left(\sum_{i\in
    I,i<i_{l}}r_{i}\right)\left(\binom{n+1}{2}-\binom{i_{l}+1}{2}\right)}|\rmN^\lri_{I\setminus\{i_{l}\},r_{I\setminus\{i_{l}\}}}(H_{i_{l}})|\]
lifts to an element in $\mathrm{N}^\lri_{I,\bfr_{I}}(H_{n})$. By \eqref{ranks
  H} we thus get
\begin{align*}
  |\rmN^\lri_{I,\bfr_{I}}(H_{n})| &
  =|\Sym_{n,n-i_{l}}(\Fq)|q^{\left(N-1\right)\left(\binom{n+1}{2}-\binom{i_{l}+1}{2}\right)}|\rmN^\lri_{I\setminus\{i_{l}\},r_{I\setminus\{i_{l}\}}}(H_{i_{l}})|\\ &
  =(q^{-2};q^{-2})^{-1}_{\lfloor(n-i_l)/2\rfloor}(q^{-i_l-1};q^{-1})_{n-i_l} q^{N\left(\binom{n+1}{2}-\binom{i_{l}+1}{2}\right)}|\rmN^\lri_{I\setminus\{i_{l}\},r_{I\setminus\{i_{l}\}}}(H_{i_{l}})|.
\end{align*}
Working recursively in this way, we obtain
\begin{align*}
  |\rmN^\lri_{I,\bfr_{I}}(H_{n})| &
  =\prod_{j=1}^{l}(q^{-2};q^{-2})^{-1}_{\lfloor\mu_{j}/2\rfloor}(q^{-i_j-1};q^{-1})_{i_{j+1}-i_j}\cdot q^{\left(\sum_{i\leq
      i_{j}}r_{i}\right)\left(\binom{i_{j+1}+1}{2}-\binom{i_{j}+1}{2}\right)}\\ &
  =\left(\prod_{j=1}^{l}(q^{-2};q^{-2})_{\lfloor\mu_{j}/2\rfloor}^{-1}\right)(q^{-i_1-1};q^{-1})_{n-i_1}\cdot q^{\sum_{i\in
      I}r_{i}\left(\binom{n+1}{2}-\binom{i+1}{2}\right)},
\end{align*}
and the proposition is proved.
\end{proof}

\begin{remark}
 An alternative approach to the proof of Proposition~\ref{pro:
   Mainthm-formulae} is to observe that suitable groups act on the
   sets $ \mathrm{N}^\lri_{I,\bfr_{I}}(\bfG)$ -- viewed as subsets of
   $\Mat_r(\lri/\mfp^N)$ -- with few orbits. For example, the group
   $\GL_{2n+\delta}(\lri)$ acts transitively on each of the sets $
   \mathrm{N}^\lri_{I,\bfr_{I}}(F_{n,\delta})$, viewed as sets of
   antisymmetric $(2n+\delta)\times(2n+\delta)$-matrices, via
   simultaneous row- and column-operations, that is via the action
   $(g,x)\mapsto gxg^\mathrm{t}$, reducing the computations of the
   numbers $|\rmN^\lri_{I,\bfr_I}(F_{n,\delta})|$ to stabiliser
   computations. A similar argument works for the groups of
   type~$G$. For groups of type $H$, however, this approach leads one
   to consider equivalence classes of quadratic forms over compact
   discrete valuation rings of characteristic zero. This is
   straightforward if the residue field characteristic is odd, but
   much more complicated if~$p=2$, obscuring the fact that the
   resulting formula~\eqref{FormulaH} holds uniformly for all~$p$. A
   similar phenomenon seems to occur in the computation of the
   integral \eqref{pvs sym} over the relative invariant of the
   prehomogeneous vector space of symmetric matrices; cf.\ the remark
   on the bottom of p.\ 177 in \cite{Igusa/00}.
\end{remark}

We now finish the proof of Theorem~\ref{thmABC:thm C}. For
$\zeta_{F_{n,\delta}(\lri)}(s)$ we obtain, by~\eqref{rewrite
zeta=poincare} and~\eqref{FormulaF},
\begin{align*}
 \zeta_{F_{n,\delta}(\lri)}(s) &
 =\sum_{I\subseteq[n-1]_{0}}\sum_{\bfr_{I}\in\N^{I}}|\rmN^\lri_{I,\bfr_{I}}(F_{n,\delta})|q^{-s\sum_{i\in
 I}r_{i}(n-i)}\\ &
 =\sum_{I\subseteq[n-1]_{0}}f_{F_{n,\delta},I}(q^{-1})\sum_{\bfr_{I}\in\N^{I}}q^{\sum_{i\in
 I}r_{i}\left(\binom{2n+\delta}{2}-\binom{2i+\delta}{2}-(n-i)s\right)}\\
 & =\sum_{I\subseteq[n-1]_{0}}f_{F_{n,\delta},I}(q^{-1})\prod_{i\in
 I}\frac{q^{\binom{2n+\delta}{2}-\binom{2i+\delta}{2}-(n-i)s}}{1-q^{\binom{2n+\delta}{2}-\binom{2i+\delta}{2}-(n-i)s}}.\end{align*}
 Similarly we obtain the following formulae for
 $\zeta_{G_{n}(\lri)}(s)$ and $\zeta_{H_{n}(\lri)}(s)$ by
 combining~\eqref{rewrite zeta=poincare} with \eqref{FormulaG}
 and~\eqref{FormulaH}, respectively:
\begin{align*} \zeta_{G_{n}(\lri)}(s) &
  =\sum_{I\subseteq[n-1]_{0}}f_{G_{n},I}(q^{-1})\prod_{i\in
    I}\frac{q^{n^{2}-i^{2}-(n-i)s}}{1-q^{n^{2}-i^{2}-(n-i)s}},\\
    \zeta_{H_{n}(\lri)}(s) &
    =\sum_{I\subseteq[n-1]_{0}}f_{H_{n},I}(q^{-1})\prod_{i\in
    I}\frac{q^{\binom{n+1}{2}-\binom{i+1}{2}-(n-i)s}}{1-q^{\binom{n+1}{2}-\binom{i+1}{2}-(n-i)s}}.\end{align*}
    This concludes the proof of Theorem~\ref{thmABC:thm C}.

\section{A multinomial-type identity and signed permutation statistics}\label{sec:binomial}
In this section we prove Proposition~\ref{pro:multinomial}, express the
polynomials $f_{F_{n,\delta},I}(X)$ and $f_{G_n,I}(X)$ defined
in~\eqref{eq:f type F} and \eqref{eq:f type G} in terms of generating
functions over descent classes in Weyl groups of type $B$, and compute
a number of joint distribution of statistics on Weyl groups of types
$B$ and~$A$. We remark that Proposition~\ref{pro:multinomial} may well
have a proof in the context of basic hypergeometric series. It
resembles, for instance, the $q$-multinomial theorem;
cf.~\cite[Exercise 1.3(ii)]{GasperRahman/04}. Lacking a suitable
reference, we prove it here directly.
\subsection{Proof of Proposition~\ref{pro:multinomial}}\label{subsec:binomial}
Recall from Section~\ref{subsec:Notation} that given a subset $I\subseteq\N$ we write $I_0$ for $I\cup\{0\}$ and for
$a,b\in\Z$ we write $aI+b=b+aI$ for the set $\{ai+b \mid i\in I\}$.
On several occasions we will use the bijections
\begin{align}
 \{I\mid I\subseteq[n-1-j]\} \label{bijections}&
 \longleftrightarrow\{I\subseteq[n-1]_{0}\mid\min\{I\cup\{n\}\}=j\}\\ I
 & \longmapsto I_{0}+j,\nonumber
\end{align}
for $j\in[n-1]_0$. We will also make use of the following, easily
verifiable identities:
\begin{align}
 \binom{n}{I_{0}+j}_{X} &
=\binom{n}{j}_{X}\binom{n-j}{I}_{X},\quad\text{for
}j\in[n-1]_0,\, I\subseteq[n-1-j],\label{eq:Binom-identity-I0+j}\\
\binom{n}{n-I}_{X} & =\binom{n}{I}_{X},\quad\text{for
}I\subseteq[n]_{0},\label{eq:Binom-identity-m-I}\\ \binom{n}{j}_X &=
\binom{n}{j}_{X^{-1}}X^{j(n-j)},\quad\text{for
}j\in[n-1]_0.\label{eq:binom-p-inv}
\end{align}
The following is known as the $q$-binomial theorem.

\begin{lemma} \label{lem:q-Chu-Vandermonde}
 For $n\in\mathbb{N}$ we have
\begin{equation*}
(ZY;X)_{n}=\sum_{j=0}^{n}\binom{n}{j}_{X}Z^{j}(Z;X)_{n-j}(Y;X)_{j}.
\end{equation*}
\end{lemma}

\begin{proof}
See, for example, \cite[Formula~1.16]{Gasper}.
\end{proof}

We begin by proving a special case of
Proposition~\ref{pro:multinomial}. We set, for $n\in\N$,
\begin{align*}
\mathcal{A}_n(X,Z) &:= \sum_{I\subseteq[n-1]}\binom{n}{I}_{X^{-1}}\prod_{i\in
I}\gp{(X^{i}Z)^{n-i}},\\
\mathcal{B}_n(X,Y,Z) &:=\sum_{I\subseteq[n-1]_0}\binom{n}{I}_{X^{-1}}(YX^{-i_1-1};X^{-1})_{n-i_1}
\prod_{i\in I}\gp{(X^{i}Z)^{n-i}}.
\end{align*}
Note that $\mathcal{A}_n(X,Z) := \mathcal{B}_n(X,0,Z)(1-Z^n)$.
\begin{proposition}\label{pro:multinomial A}
For $n\in\N$ we have
\begin{equation}\label{equ:binomial A}
\mathcal{A}_n(X,Z)=\frac{1-Z^{n}}{(Z;X)_{n}}=\sum_{I\subseteq[n-1]}\binom{n}{I}_{X^{-1}}\prod_{i\in
I}\gp{(X^{n-i}Z)^{i}}.
\end{equation}
\end{proposition}
\begin{proof}
We prove the first equation by induction on~$n$.  For $n=1$, we have\[
\mathcal{A}_{1}(X,Z)=\binom{1}{\varnothing}_{X^{-1}}=\frac{1-Z}{(Z;X)_{1}}=1.\]
Suppose now that $n>1$ and that the assertion holds for all~$m<n$. The
key idea for the proof is to re-organise the sum defining
$\mathcal{A}_n$ according to the minima of the indexing subsets. Using
the bijections~\eqref{bijections}, identity
\eqref{eq:Binom-identity-I0+j}, the induction hypothesis and
identity~\eqref{eq:binom-p-inv} we obtain
\begin{align*}
\mathcal{A}_n(X,Z) &
 =\sum_{j=1}^{n}\sum_{\substack{I\subseteq[n-1]\\ \min\{I\cup\{n\}\}=j}
 }\binom{n}{I}_{X^{-1}}\prod_{i\in I}\gp{(X^{i}Z)^{n-i}}\\ &
 =1+\sum_{j=1}^{n-1}\sum_{I\subseteq[n-1-j]}\binom{n}{I_{0}+j}_{X^{-1}}\prod_{i\in
   I_{0}+j}\gp{(X^{i}Z)^{n-i}}\\ &
 =1+\sum_{j=1}^{n-1}\binom{n}{j}_{X^{-1}}\sum_{I\subseteq[n-1-j]}\binom{n-j}{I}_{X^{-1}}\prod_{i\in
   I_{0}+j}\gp{(X^{i}Z)^{n-i}}\\ &
 =1+\sum_{j=1}^{n-1}\binom{n}{j}_{X^{-1}}\gp{(X^{j}Z)^{n-j}}\sum_{I\subseteq[n-1-j]}\binom{n-j}{I}_{X^{-1}}\prod_{i\in
   I}\gp{(X^{i+j}Z)^{n-i-j}}\\ &
 =1+\sum_{j=1}^{n-1}\binom{n}{j}_{X^{-1}}\gp{(X^{j}Z)^{n-j}}\mathcal{A}_{n-j}(X,X^jZ)\\ &
 =1+\sum_{j=1}^{n-1}\binom{n}{j}_{X}X^{-j(n-j)}\gp{(X^{j}Z)^{n-j}}\frac{1-(X^{j}Z)^{n-j}}{(X^{j}Z;X)_{n-j}}\\ &
 =\sum_{j=1}^{n}\binom{n}{j}_X\frac{Z{}^{n-j}}{(X^{j}Z;X)_{n-j}}.
\end{align*}
Writing this expression for $\mathcal{A}_n(X,Z)$ on the common
denominator~$(Z;X)_{n}$, we obtain
\[
 \mathcal{A}_n(X,Z)(Z;X)_{n}=\sum_{j=1}^{n}\binom{n}{j}_XZ{}^{n-j}(Z;X)_{j}=-Z^{n}+\sum_{j=0}^{n}\binom{n}{j}_XZ{}^{n-j}(Z;X)_{j}.
\]
Changing $j$ to $n-j$ and applying Lemma~\ref{lem:q-Chu-Vandermonde}
with $Y=0$ yields
\[
\sum_{j=0}^{n}\binom{n}{j}_XZ{}^{n-j}(Z;X)_{j}=\sum_{j=0}^{n}\binom{n}{j}_XZ{}^{j}(Z;X)_{n-j}=1.\]
Thus
\[ \mathcal{A}_n(X,Z)=\frac{1-Z^{n}}{(Z;X)_{n}}.\]
The second equation in~\eqref{equ:binomial A} follows from the first
equation, by changing $i$ to $n-i$ and using
\eqref{eq:Binom-identity-m-I}.
\end{proof}

We now prove Proposition~\ref{pro:multinomial} in general. We organise
the sum defining $\mathcal{B}_n$ according to the minima of the
indexing subsets.  For {$j\in[n]_{0}$,} let\[
S_{n,j}(X,Z):=\sum_{\substack{I\subseteq[n-1]_{0}\\
\min\{I\cup\{n\}\}=j} }\binom{n}{I}_{X^{-1}}\prod_{i\in
I}\gp{(X^{i}Z)^{n-i}}.\] We claim that, for all $j\in[n]_0$,
\begin{equation}\label{equ:S}
 S_{n,j}(X,Z)=\binom{n}{j}_{X}\frac{Z^{n-j}}{(X^{j}Z;X)_{n-j}}.\end{equation}
This clearly holds for $j=n$, so assume $j\in[n-1]_0$.  Due to the
bijections~\eqref{bijections} and the identities
\eqref{eq:Binom-identity-I0+j}, \eqref{eq:Binom-identity-m-I}, \eqref{equ:binomial A} and \eqref{eq:binom-p-inv}, we have
\begin{align*}
 S_{n,j}(X,Z) &
 =\sum_{I\subseteq[n-1-j]}\binom{n}{I_0+j}_{X^{-1}}\prod_{i\in
   I_0+j}\gp{(X^{i}Z)^{n-i}}\\ &
 =\binom{n}{j}_{X^{-1}}\sum_{I\subseteq[n-1-j]}\binom{n-j}{I}_{X^{-1}}\prod_{i\in
   I_0+j}\gp{(X^{i}Z)^{n-i}}\\ &
 =\binom{n}{j}_{X^{-1}}\gp{(X^{j}Z)^{n-j}}\mathcal{A}_{n-j}(X,X^{j}Z)\\ &
 =\binom{n}{j}_{X^{-1}}\gp{(X^{j}Z)^{n-j}}\frac{1-(X^{j}Z)^{n-j}}{(X^{j}Z;X)_{n-j}}\\ &
 =\binom{n}{j}_{X}\frac{Z^{n-j}}{(X^{j}Z;X)_{n-j}},\end{align*}
establishing~\eqref{equ:S}. This yields
\begin{align}
 \mathcal{B}_n(X,Y,Z) &=
\sum_{j=0}^n(YX^{-j-1};X^{-1})_{n-j}S_{n,j}(X,Z)\nonumber\\&=\sum_{j=0}^n\binom{n}{j}_{X}(YX^{-j-1};X^{-1})_{n-j}\frac{Z^{n-j}}{(X^{j}Z;X)_{n-j}}
\label{equ:sum F}.
\end{align}
Writing this expression for $\mathcal{B}_n(X,Y,Z)$ on the common
denominator $(Z;X)_n$, we obtain
\begin{equation}
 \mathcal{B}_n(X,Y,Z)(Z;X)_{n}\label{eq:Numerator-F}
 =\sum_{j=0}^{n}\binom{n}{j}_{X}(YX^{-j-1};X^{-1})_{n-j}Z^{n-j}(Z;X)_{j}.
\end{equation}
Using the identities
$$ (YX^{-j-1};X^{-1})_{n-j}=(X^{-n}Y;X)_{n-j},
$$ for $j\in[n-1]_0$, changing $j$ to $n-j$ and applying
Lemma~\ref{lem:q-Chu-Vandermonde} with $Y$ replaced by $X^{-n}Y$, we
can rewrite the right-hand side of \eqref{eq:Numerator-F} as\[
\sum_{j=0}^{n}\binom{n}{j}_{X}Z^j(Z;X)_{n-j}(X^{-n}Y;X)_{j}=(X^{-n}YZ;X)_{n}.\]
This proves Proposition~\ref{pro:multinomial}.

\begin{remark}\label{rem:subgroup zeta}
Given a $\T$-group $G$, its \emph{subgroup zeta function} is defined
as the Dirichlet series
$$\zeta^<_G(s) := \sum_{H\leq_f G} |G:H|^{-s},$$ where $s$ is a
complex variable and the sum ranges over the subgroups of $G$ of
finite index; cf.~\cite[Chapter 15]{LubotzkySegal/03}. It is well
known that the zeta function of $G=\Z^n$ equals
$$\zeta^<_{\Z^n}(s) = \prod_{i=0}^{n-1}\zeta(s-i)=\prod_{p \text{
    prime}}\frac{1}{(p^{-s};p)_n}=\prod_{p \text{
    prime}}\mathcal{B}_n(p,0,p^{-s}),$$ where $\zeta(s)$ is the
Riemann zeta function; see, for instance,
\cite[Theorem~51.1]{LubotzkySegal/03}. The expression of the local
zeta function of $\zeta^<_{\Z^n}(s)$ in terms of a sum, like the one
defining the function $\mathcal{B}_n$, illustrates a general approach
to the study of local (subgroup and representation) zeta functions of
$\T$-groups developed in~\cite{Voll/10}.
\end{remark}

\subsection{Some Weyl group generating functions} \label{subsec:weyl}
Our main source for background material on Coxeter groups
is~\cite{BjoernerBrenti/05}.  Let $(W,S)$ be a finite Coxeter system,
consisting of a finite Coxeter group $W$ and a set $S$ of Coxeter
generators for~$W$. For $w\in W$, the length of $w$, denoted
by~$l(w)$, is the minimal length of a word in elements of $S$
representing~$w$.  Recall that the (right) descent set of $w$ is
defined as
\begin{equation*}
D(w):=\{s\in S \mid \;l(ws)<l(w)\}.
\end{equation*}
For $I\subseteq S$, we denote by $W_I=\langle I\rangle$ the
corresponding standard parabolic subgroup of~$W$. We also have the
so-called quotient
\begin{equation}\label{def:quotient}
W^I:=\{w\in W \mid \;D(w)\subseteq I^c\}.
\end{equation}
The quotient $W^I$ is the collection of the unique coset
representatives of $W_I$ of shortest length.

Consider now, specifically, Weyl groups of type~$B$. Let
$n\in\N$. Recall that we defined the group $B_n$ as the group of all
bijections $w$ of the set $[\pm n]_0$ such that, for all $a\in[\pm
n]_0$, $w(-a)=-w(a)$. Such bijections are determined by their values
on the positive integers up to~$n$, and thus $B_n$ may be viewed as
the group of signed permutations, that is monomial matrices with
non-zero entries in~$\{-1,1\}$. We write $w=[a_1,\dots,a_n]$ to mean
that, for $i\in[n]$, $w(i)=a_i$. By $S=\{s_0,s_1,\dots,s_{n-1}\}$ we
denote the set of standard Coxeter generators of~$B_n$, that is
$s_i=[1,\dots,i-1,i+1,i,i+2,\dots,n]$ for $i\in[n-1]$ and
$s_0=[-1,2,\dots,n]$.  We frequently identify $S$ with the interval
$[n-1]_0$ in the obvious way.  Given $w\in B_n$, we define
$\coxneg(w):= \#\{i\in[n] \mid\;w(i)<0\}$.

\begin{lemma}\label{lem:Common-denominator}
 Let $(W,S)$ be a finite Coxeter system, $r\in\N$, and let
 $Y_1,\dots,Y_r$ and $Z_i, i\in S$, be independent variables. Let
 $h_w(\bfY)=h_w(Y_1,\dots,Y_r)$, $w\in W$, be polynomials in
 $\mathbb{Q}[Y_1,\dots,Y_r]$.  Then the following identity holds:
\[ \sum_{I\subseteq S}\left(\sum_{w\in
   W^{I^c}}h_{w}(\bfY)\right)\prod_{i\in
   I}\frac{Z_{i}}{1-Z_i}=\frac{\sum_{w\in W}h_{w}(\bfY)\prod_{i\in
   D_W(w)}Z_{i}}{\prod_{i\in S}(1-Z_{i})}.\]
 \end{lemma}
\begin{proof}
This is an easy application of the inclusion-exclusion principle.
\end{proof}

\subsubsection{Joint distribution of $(l,\nega)$ over descent classes of $B_n$}\label{subsec:proof of prop}

\begin{lemma}[Reiner]\label{lem:reiner}
 For $n\in\N$ and $I=\{i_1,\dots,i_l\}_<\subseteq[n-1]_0$ we have
\begin{equation}\label{equ:reiner}
\sum_{w\in B _n^{I^{c}}}X^{l(w)}Y^{\coxneg(w)}=\binom{n}{I}_{X}(-YX^{i_1+1};X)_{n-i_1}.
\end{equation}
\end{lemma}
\begin{proof}
This is proved by Reiner in \cite[Lemma~3.1]{Reiner-Signed-perm}; we
just need to translate between Reiner's notation and ours.  First,
note that Reiner uses `$\coxinv$' to denote the length function $l$ on
$B_n$. Recall that we set $i_{0}=0$, and, for $j\in\{0,1,\dots,l\}$,
\begin{equation}\label{def:mu}
\mu_{j} = i_{j+1}-i_{j}.
\end{equation} The relations~\eqref{def:mu} provide the transition between
Reiner's sets
$$\mathcal{S}=\{\mu_{k},\mu_{k}+\mu_{k-1},\dots,\mu_{k}+\dots+\mu_{1}\}\subseteq[n]$$
and our sets $I\subseteq[n-1]_0$. Moreover, $\mu_0=i_1$. Given
$\mathcal{S}\subseteq[n]$, an element $w\in B_n$ satisfies Reiner's
relation `$D(w)\subseteq \mathcal{S}$' if and only if it satisfies
$D(w)\subseteq I$, which is, by~\eqref{def:quotient}, equivalent to
$w\in B_n^{I^{c}}$.  From \cite[Lemma~3.1]{Reiner-Signed-perm} and
formula (2) in its proof we thus obtain, partly in the notation
of~\cite{Reiner-Signed-perm},
\begin{multline*}
\sum_{w\in B _n^{I^{c}}}X^{l(w)}Y^{\coxneg(w)} =\frac{[\hat{n}]!_{Y,X}}{[\hat{\mu}_{0}]!_{Y,X}[\mu_{1}]!_{X}\cdots[\mu_{l}]!_{X}}
 =\frac{(-XY;X)_{n}[n]!_{X}}{(-XY;X)_{\mu_{0}}[\mu_{0}]_{X}![\mu_{1}]!_{X}\cdots[\mu_{l}]!_{X}}\\=(-YX^{i_1+1};X)_{n-i_1}\frac{[n]!_{X}}{[i_{1}]_{X}![i_{2}-i_{1}]!_{X}\cdots[n-i_l]!_{X}}
 =\binom{n}{I}_{X}(-YX^{i_1+1};X)_{n-i_1}.\end{multline*}
\end{proof}

\begin{proposition}\label{pro:combinatorial formulae}
 Let $n\in\N$, $\delta\in\{0,1\}$ and let $I\subseteq[n-1]_{0}$. The
 polynomials $f_{F_{n,\delta},I}(X)$ and $f_{G_{n},I}(X)$ defined
 in~\eqref{eq:f type F} and \eqref{eq:f type G} satisfy the following
 identities:
 \begin{align*}
  f_{F_{n,\delta},I}(X) &= \sum_{w\in
      B_{n}^{I^{c}}}(-1)^{\coxneg(w)}X^{(2l+(2\delta-1)\coxneg)(w)},\label{equ:combinatorics
      F} \\ f_{G_{n},I}(X) &= \sum_{w\in
      B_{n}^{I^{c}}}(-1)^{\coxneg(w)}X^{l(w)}.\nonumber
 \end{align*}
\end{proposition}

\begin{proof}
Replace $(X,Y)$ in~\eqref{equ:reiner} by $(X^2,-X^{2\delta-1})$ in type
$F$ and by $(X,-1)$ in type~$G$.
\end{proof}

\subsubsection{Proof and discussion of Proposition~\ref{pro:distribution}}\label{subsec:proof of thm D}

We recall that by Proposition~\ref{pro:multinomial} we have
$$\mathcal{B}_n(X,Y,Z) :=
\sum_{I\subseteq[n-1]_0}\binom{n}{I}_{X^{-1}}(YX^{-i_1-1};X^{-1})_{n-i_1}
\prod_{i\in I}\gp{(X^{i}Z)^{n-i}}=\frac{(X^{-n}YZ;X)_n}{(Z;X)_n}$$ and
by Lemma~\ref{lem:reiner} we have, for all $I\subseteq[n-1]_0$,
$$\binom{n}{I}_{X^{-1}}(-YX^{-i_1-1};X^{-1})_{n-i_1} = \sum_{w\in
  B_n^{I^c}}X^{-l(w)}Y^{\coxneg(w)}.$$ Therefore
Lemma~\ref{lem:Common-denominator}, with
$(W,S)=(B_{n},\{s_0,\dots,s_{n-1}\})$, implies that
$$\mathcal{B}_n(X,-Y,X^nZ) = \frac{\sum_{w\in
B_n}X^{-l(w)}Y^{\coxneg(w)}\prod_{i\in
D(w)}(X^{n+i}Z)^{n-i}}{\prod_{i=0}^{n-1}(1-(X^{n+i}Z)^{n-i})}=
\prod_{i=0}^{n-1}\frac{1+X^{i}YZ}{1-X^{n+i}Z}.$$ Hence
\begin{equation}\label{equ:distribution B}
\sum_{w\in B_n}X^{(\sigma-l)(w)}Y^{\coxneg(w)}Z^{\rmaj(w)} =
\prod_{i=0}^{n-1}\frac{(1+X^{i}YZ)(1-(X^{n+i}Z)^{n-i})}{1-X^{n+i}Z},
\end{equation}
concluding the proof of Proposition~\ref{pro:distribution}.

\smallskip

We see Proposition~\ref{pro:distribution} in the context of a number
of results in the literature which establish multivariate generating
functions describing the joint distributions of various statistics on
finite Weyl groups, sometimes `twisted' by $1$-dimensional
representations; see, for example, \cite{Reiner-Signed-perm,
Reiner/95, Biagioli/06}.  For instance we observe that setting $X=1$
in~\eqref{equ:distribution B} yields a special case of
\cite[Theorem~3.2]{Reiner/95}. Upon setting $Y=Z=1$ in
\eqref{equ:distribution B} we recover
\cite[Theorem~1.1]{StembridgeWaugh/98} for Weyl groups of type~$B$. In
its generality, the latter result describes the generating function
$\sum_{w\in W}X^{(\sigma-l)(w)}$ for a finite Weyl group~$W$ in terms
of the simple root coordinates $b_i$ and the Weyl group's exponents. A
twisted version of this result for Weyl groups of type $B$ is the
following.

\begin{corollary}
$$\sum_{w\in B_n}(-1)^{\coxneg(w)}X^{(\sigma-l)(w)}=0.$$
\end{corollary}
\begin{proof}
Set $Y=-1$ and $Z=1$ in~\eqref{equ:distribution B}.
\end{proof}

Analysing our formulae for $\mathcal{A}_n(X,Z)$ yields formulae over
Weyl groups of type $A$ which are similar to~\eqref{equ:distribution
  B}. In the case of the Weyl group $W=S_n$, with Coxeter generating
set $S=(s_1,\dots,s_{n-1})$ comprising the standard transpositions,
the simple root coordinates $b_i$, $i\in[n-1]$, defined in
Section~\ref{sec:intro} are given by $b_i=i(n-i)$;
cf.~\cite[Remark~1.5]{StembridgeWaugh/98} or
\cite[Plate~I]{Bourbaki/02}. Therefore
\begin{equation}\label{equ:root coordinates S}
\sigma(w) = \sum_{i\in D(w)}b_i = \sum_{i\in D(w)}i(n-i),\quad\text{
  for $w\in S_n$}.
\end{equation}
Here we identified the generating set $S$ with the interval $[n-1]$ in
the obvious way. The statistics $\maj$ and $\rmaj$ on $S_n$ are
defined by setting, for $w\in S_n$, $\maj(w)=\sum_{i\in D(w)}i$ and
$\rmaj(w)=\sum_{i\in D(w)}(n-i)$, respectively.

\begin{proposition} \label{pro:S_n}
\begin{equation}
\sum_{w\in S_n} X^{(\sigma-l)(w)}Z^{\maj(w)}=\sum_{w\in S_n}
X^{(\sigma-l)(w)}Z^{\rmaj(w)}=
\prod_{i=0}^{n-1}\frac{1-(X^iZ)^{n-i}}{1-X^iZ}\label{equ:distribution
  S_n}.
\end{equation}
\end{proposition}

\begin{proof}
By~\eqref{equ:binomial A} we have
 \[
  \mathcal{A}_n(X,Z)=\sum_{I\subseteq[n-1]}\binom{n}{I}_{X^{-1}}\prod_{i\in
    I}\gp{X^{i(n-i)}Z^{i}}=\frac{1-Z^{n}}{(Z;X)_{n}},\] and by\
    \cite[Proposition~1.3.17]{Stanley/97} we have, for
    $I\subseteq[n-1]$,
 \[
  \binom{n}{I}_{X^{-1}}=\sum_{w\in S_{n}^{I^c}}X^{-l(w)}.
 \]
 Therefore Lemma~\ref{lem:Common-denominator}, with
 $(W,S)=(S_{n},\{s_1,\dots,s_{n-1}\})$, implies that
 \[
  \mathcal{A}_n(X,Z)=\frac{\sum_{w\in S_{n}}X^{-l(w)}\prod_{i\in
      D(w)}X^{i(n-i)}Z^{i}}{\prod_{i=1}^{n-1}(1-(X^{n-i}Z)^{i})}=\frac{\sum_{w\in
      S_{n}}X^{(\sigma-l)(w)}Z^{\maj(w)}}{\prod_{i=1}^{n-1}(1-(X^{n-i}Z)^{i})},
 \]
 and so
 \begin{equation*}
  \sum_{w\in S_{n}}X^{(\sigma-l)(w)}Z^{\maj(w)}
  =\mathcal{A}_n(X,Z)\prod_{i=1}^{n-1}(1-(X^{n-i}Z)^{i})
  =\prod_{i=0}^{n-1}\frac{1-(X^{i}Z)^{n-i}}{1-X^{i}Z}.
 \end{equation*}
The equality involving $\rmaj$ follows similarly, using the second
equality in~\eqref{equ:binomial A}.
\end{proof}

Note that setting $X=1$ in~\eqref{equ:distribution S_n} yields the
Poincar\'e polynomial of $S_{n}$, reflecting the well-known facts that
the statistics $\maj$ and $\rmaj$ on $S_n$ are Mahonian, that is
equidistributed with the length function~$l$. Setting $Z=1$
reproduces~\cite[Remark~1.5]{StembridgeWaugh/98}.

\section{Proof of Theorem~\ref{thmABC:thm B}}\label{sec:proof thm B}

\subsection{Proof of Theorem~\ref{thmABC:thm B}}\label{subsec:thm B}
We start by proving the formulae for the zeta functions of groups of
type $F$ and $G$ given in \eqref{equ:mult F} and~\eqref{equ:mult
G}. Considering the Euler product~\eqref{equ:euler}, it clearly
suffices to establish the following result.

\begin{proposition}\label{pro:thm B type F and G}
For every non-zero prime ideal $\mfp$ of $\Gri$, with $|\Gri:\mfp|=q$,
say, we have
\begin{align}
\zeta_{F_{n,\delta}(\Gri_\mfp)}(s) &=
\frac{(q^{-s};q^{2})_{n}}{(q^{2(n+\delta)-1-s};q^{2})_{n}},\label{pro:Mult-formula-F}\\ \zeta_{G_{n}(\Gri_\mfp)}(s)
&= \frac{(q^{-s};q)_{n}}{(q^{n-s};q)_n}. \label{pro:Mult-formula-G}
\end{align}
\end{proposition}

\begin{proof} By Theorem~\ref{thmABC:thm C}, we have
\begin{align*}
\zeta_{F_{n,\delta}(\Gri_\mfp)}(s)&=\sum_{I\subseteq[n-1]_{0}}\binom{n}{I}_{q^{-2}}(q^{-2(i_{1}+\delta)-1};q^{-2})_{n-i_{1}}\prod_{i\in
  I}\gp{q^{(2(n+i+\delta)-1-s)(n-i)}},\\ \zeta_{G_{n}(\Gri_\mfp)}(s) &=
\sum_{I\subseteq[n-1]_{0}}\binom{n}{I}_{q^{-1}}(q^{-i_1-1};q^{-1})_{n-i_1}\prod_{i\in
  I}\gp{q^{(n+i-s)(n-i)}}.
\end{align*}
Thus
$\zeta_{F_{n,\delta}(\Gri_\mfp)}(s)=\mathcal{B}_n(q^2,q^{-2\delta+1},q^{2(n+\delta)-1-s})$
and $\zeta_{G_{n}(\Gri_\mfp)}(s)=\mathcal{B}_n(q,1,q^{n-s})$, and the claim
follows from Proposition~\ref{pro:multinomial}.
\end{proof}

The rest of this section is dedicated to proving the formulae for the
zeta functions of groups of type $H$ given in~\eqref{equ:mult
  H}. Short of direct proof akin to the proof of
Proposition~\ref{pro:thm B type F and G}, we reduce type $H$ to
type~$F$; cf.\ Proposition~\ref{pro:H=F}. For this result we need some
preparation.

Recall that $n=2m+\eps\in\N$ with $\eps\in\{0,1\}$, that we write
$I=\{i_1,\dots,i_l\}_<\subseteq[n-1]_0$ and
$J=\{j_1,\dots,j_k\}\subseteq[m-1]_0$ for subsets of $[n-1]_0$ and
$[m-1]_0$, respectively, and the conventions that $i_0=0$
and~$i_{l+1}=n$. We set
\begin{equation*}
 f_{n,I} := f_{H_{n},I}(q^{-1}) :=
 \left(\prod_{j=1}^{l}(q^{-2};q^{-2})_{\lfloor
   (i_{j+1}-i_j)/2\rfloor}^{-1}\right)(q^{-i_1-1};q^{-1})_{n-i_1}\quad\textrm{
   for $I\subseteq[n-1]_0$,}
\end{equation*}
  and
\begin{equation}
  X_i := X_i(H_{n}):= q^{\binom{n+1}{2} -
  \binom{i+1}{2}-(n-i)s}\quad\text{ for }i\in[n-1]_0.\label{def:X_i}
\end{equation}
Given $I\subset[n-1]_0$ we write $\Pi_I$ for $\prod_{i\in
I}\gp{X_i(H_{n})}$.

Theorem~\ref{thmABC:thm C} represents the local factor
$\zeta_{H_{n}(\Gri_\mfp)}(s)$ as a sum, indexed by the subsets
of~$[n-1]_0$. It is advantageous to organise this sum according to the
fibres of the surjective `bisection map'~$\phi$, defined as
follows. We set
\begin{equation*}
\phi:2^{[n-1]_0}\longrightarrow 2^{[m-1]_0},\quad\quad
I\longmapsto \left\{ \lfloor \frac{i+1-\eps}{2}\rfloor \mid i\in I\right\} \setminus\{m\}.
\end{equation*}
We note that the `doubling map' $2^{[m-1]_0}\rightarrow 2^{[n-1]_0},
J\mapsto 2J+\eps$, is a section of this map, so each fibre
$\phi^{-1}(J)$ contains the set $2J+\eps\subseteq[n-1]_0$.  The proof
of Proposition~\ref{pro:H=F} hinges on the following technical
result. Its proof will occupy the bulk of this section.

\begin{lemma}\label{lem:fibres}
 For all $J\subseteq[m-1]_0$ we have
\begin{equation}\label{equ:lem fibres}
 \sum_{I\in\phi^{-1}(J)} f_{I} \Pi_I = \left( 1 +
(\udrl{n})\gp{X_{n-1}}\right)f_{ 2J+\eps} \Pi_{2J+\eps} =
\frac{1-q^{-s}}{1-q^{n-s}} f_{ 2J+\eps} \Pi_{2J+\eps} .
\end{equation}
\end{lemma}

Informally speaking, Lemma~\ref{lem:fibres} `eliminates' occurrences
of the terms $X_i$, where $i\in[n-2]_0\setminus (2[m-1]_0+\eps)$, and
simplifies the sum on the left hand side of~\eqref{equ:lem fibres}, which has roughly $3^{|J|}$ terms.

\begin{proof}
The second equation in~\eqref{equ:lem fibres} is clear. The proof of
the first equation requires slightly different arguments in the cases
$\eps=0$ and $\eps=1$. Nevertheless we treat both cases in
parallel. We start with an observation in the case $\eps=0$. Let
$J\subseteq[m-1]_0$. We note that, if $\eps=0$, we have $0\in J$ if
and only if $0\in I$ for all $I\in\phi^{-1}(J)$. We claim that we may,
without loss of generality, assume in this case that $0\not\in
J$. Indeed, if $0\in J$, we write $J=\{0\}\dotcup J'$, where
$J'=J\setminus\{0\}$, and similarly, for each $I\in\phi^{-1}(J)$, we
write $I=\{0\}\dotcup I'$, where $I'=I\setminus\{0\}$. Set
$j'_1:=\min\{J'\cup\{m\}\}$. It now suffices to observe that for all
$I\subseteq\phi^{-1}(J)$ (including the set $I=2J$) one has
$f_{I}=(q^{-1};q^{-2})_{j_1'}f_{I'}$. Therefore if \eqref{equ:lem
  fibres} holds for $J'$ it also holds for $J$. Indeed, we then have
\begin{align}\label{equ:0}
\sum_{I\in\phi^{-1}(J)} f_{I} \Pi_I &=
(q^{-1};q^{-2})_{j_1'}\;\gp{X_0}\sum_{I\in\phi^{-1}(J')} f_{I}
\Pi_I\nonumber\\&=(q^{-1};q^{-2})_{j_1'}\;\gp{X_0}\left(1 +
(\udrl{n})\gp{X_{n-1}}\right)f_{2J'}\Pi_{2J'}\nonumber\\
&=\left(1 + (\udrl{n})\gp{X_{n-1}}\right)f_{2J}\Pi_{2J}.
\end{align}
We thus assume henceforth that $0\not\in J$ if $\eps=0$.

We return to the general situation with $\eps\in\{0,1\}$. A key role
in the proof is played by the relations
\begin{equation}\label{relations}
X_{2j+\eps-1} = q^{-2(m-j)} X_{2j+\eps}X_{n-1}, \quad
j\in\{1-\eps,\dots,m-1\},\end{equation} which are immediate from the
definitions~\eqref{def:X_i}. The validity of Lemma~\ref{lem:fibres}
depends only on these relations, and not on the particular `numerical
data'~$(X_i)$. The first equation of~\eqref{equ:lem fibres} is thus
equivalent to an equality in the quotient of the ring
$\Z[q^{-1},X_1,\dots,X_{n-1}]$ by the ideal generated by the
relations~\eqref{relations}. The independence from the precise
numerical data $(X_i)$ is used in an inductive argument later in the
proof.

We prove Lemma~\ref{lem:fibres} by induction on $|J|$. We first
deal with the special case $J=\varnothing$, the base for our
induction. It is clear that
$\phi^{-1}(\varnothing)=\{\varnothing,\{n-1\}\}$ and easily checked
that $f_{\varnothing}=1$ and $f_{\{n-1\}}=(\udrl{n})$, so that
\begin{equation}\label{equ:fibre of empty}
\sum_{I\in \phi^{-1}(\varnothing)}f_{I} \Pi_I = f_{\varnothing} +
f_{\{n-1\}}\gp{X_{n-1}}=1+(\udrl{n})\gp{X_{n-1}}
\end{equation}
as claimed.

Assume now that $k=|J|\geq 1$ and write $J=\{j_1,\dots,j_k\}_<$. Note
that, by assumption, $j_1>0$ if~$\eps=0$. Our strategy is to split up
the fibre $\phi^{-1}(J)$ into three disjoint sets of equal size
$2\cdot 3^{k-1}$, according to the intersection
with~$T_1:=\{2j_1+\eps-1,2j_1+\eps\}$. We define
\begin{align*}
\calI_1 &= \{I\in\phi^{-1}(J) \mid I \cap T_1 = \{2j_1+\eps\}\},\\
\calI_2 &= \{I\in\phi^{-1}(J) \mid I \cap T_1 =
\{2j_1+\eps-1,2j_1+\eps\}\},\\ \calI_3 &=\{I\in\phi^{-1}(J) \mid I
\cap T_1 = \{2j_1+\eps-1\}\}.
\end{align*}
Hence $\phi^{-1}(J) = \calI_1 \dotcup \calI_2 \dotcup \calI_3$, as
$I\cap T_1\neq \varnothing$ for all~$I\in\phi^{-1}(J)$.  For
$r\in\{1,2,3\}$ we set $\calS_r = \sum_{I\in\calI_r}f_{I}\Pi_I$ so
that $\sum_{I\in\phi^{-1}(J)}f_I\Pi_I = \calS_1 + \calS_2 + \calS_3$.
We claim that
 \begin{align}
 \calS_1 &= f_{2J+\eps} \Pi_{2J+\eps} \left(1 +
 (\udrl{2(m-j_1)})\gp{X_{n-1}}\right),\label{equ:S1}\\ \calS_2 &=
 (\udrl{2j_1+\eps})\gp{X_{2j_1+\eps-1}}\calS_1,\label{equ:S2}\\ \calS_3
 &= f_{2J+\eps}(\udrl{2j_1+\eps}) \gp{X_{2j_1+\eps-1}}
 \Pi_{2(J\setminus\{j_1\})+\eps}\left(1 +
 \gp{X_{n-1}}\right).\label{equ:S3}
\end{align}
To prove~\eqref{equ:S1} we note that for all $I\in\calI_1$ we have
$$f_{{n},I} = \binom{n}{2j_1+\eps}_{q^{-1}}
f_{{2(m-j_1)},I-2j_1-\eps},$$ and hence
\begin{align*}
\calS_1 &= \binom{n}{2j_1+\eps}_{q^{-1}}\sum_{I\in\calI_1}f_{{2(m-j_1)},I-2j_1-\eps}\Pi_I\\
&=\binom{n}{2j_1+\eps}_{q^{-1}}
f_{{2(m-j_1)},2J-2j_1}\Pi_{2J+\eps}(1 + (\udrl{2(m-j_1)})\gp{X_{n-1}})\\
&=f_{2J+\eps} \Pi_{2J+\eps} \left(1 +
(\udrl{2(m-j_1)})\gp{X_{n-1}}\right).
\end{align*}
Here, the second equality uses the induction hypothesis for $J -
j_1\subseteq [m-j_1-1]_0$, $\eps=0$. This is justified as the terms
$X_i(H_{n})$ for $2j_1+\eps\leq i < n$ satisfy the same relations
given by \eqref{relations} as the terms $X_i(H_{2(m-j_1)})$ for $0\leq
i < 2(m-j_1)$, and because $|(J-j_1)\cap\N|<k$.

To prove~\eqref{equ:S2} it suffices to observe that
$f_{I}=(\udrl{2j_1+\eps})f_{I\setminus\{2j_1+\eps-1\}}$ for all $I\in
\calI_2$.

To prove~\eqref{equ:S3} we proceed by a second induction on $|J|=k$,
the induction base $k=1$ being a straightforward computation which we
leave to the reader. If $k>1$ we partition the set $\calI_3$. We let
$T_2=\{2j_2+\eps-1,2j_2+\eps\}$ and define
\begin{align*}
\calI_{3,1} &= \{I\in\calI_3 \mid I\cap T_2 = \{2j_2+\eps-1\}\},\\
\calI_{3,2} &= \{I\in\calI_3 \mid I\cap T_2 = \{2j_2+\eps-1,2j_2+\eps\}\},\\
\calI_{3,3} &=\{I\in\calI_3 \mid I\cap T_2 = \{2j_2+\eps\}\}.
\end{align*}
Note that $I\cap T_2\neq\varnothing$ for all $I\in\mathcal{I}_3$.  For
$r\in\{1,2,3\}$ we set $\calS_{3,r} =
\sum_{I\in\calI_{3,r}}f_{I}\Pi_I$ so that $\calS_3 = \calS_{3,1} +
\calS_{3,2} + \calS_{3,3}$. We claim that
\begin{align}
\calS_{3,1} &= f_{2J+\eps} (\udrl{2j_1+\eps}) \gp{X_{2j_1+\eps-1}}
\Pi_{2(J\setminus\{j_1\})+\eps} \left( 1 + (\udrl{2(m-j_2)})\gp{X_{n-1}}
\right),\label{equ:S31}\\ \calS_{3,2} &= \gp{X_{2j_2+\eps-1}}
\calS_{3,1},\label{equ:S32} \\ \calS_{3,3} &= f_{2J+\eps}
(\udrl{2j_1+\eps}) \gp{X_{2j_1+\eps-1}} \,\gp{X_{2j_2+\eps-1}}
\Pi_{2(J\setminus\{j_1,j_2\})+\eps} \left( 1 + \gp{X_{n-1}}
\right).\label{equ:S33}
\end{align}

To prove \eqref{equ:S31} we observe that for all $I\in\calI_{3,1}$ we
have
\begin{align}\label{f_I eps=0}
f_{I} &=
f_{I\setminus\{2j_1-1\}}(q^{-2j_1-1};q^{-2})_{j_2-j_1}(\udrl{2{j_1}})\binom{j_2}{j_1}_{q^{-2}}\quad\text{
if $\eps=0$},\\ f_{I} &=
f_{I\setminus\{2j_1\}}(q^{-2j_1-1};q^{-2})_{j_2-j_1}(\udrl{2{j_2}+1})\binom{j_2}{j_1}_{q^{-2}}\quad\text{
if $\eps=1$}.\label{f_I eps=1}
\end{align}
Furthermore, for all $J\subseteq[m-1]_0$ we have
\begin{align}f_{2J} &= f_{2J\setminus
  2\{j_1\}}(q^{-2j_1-1};q^{-2})_{j_2-j_1}\binom{j_2}{j_1}_{q^{-2}},\label{equ:f eps=0}\\
f_{2J+1} &= f_{2(J\setminus
  2\{j_1\})+1}(q^{-2j_1-1};q^{-2})_{j_2-j_1}\binom{j_2}{j_1}_{q^{-2}}\frac{(\udrl{2j_2+1})}{(\udrl{2j_1+1})},\label{equ:f eps=1}
\end{align}
so that if $\eps=0$ we have, using \eqref{f_I eps=0}, \eqref{equ:S1}
for $J\setminus\{j_1\}$, and \eqref{equ:f eps=0},
\begin{align*}
\calS_{3,1} &=
(q^{-2j_1-1};q^{-2})_{j_2-j_1}(\udrl{2{j_1}})\binom{j_2}{j_1}_{q^{-2}}\gp{X_{2j_1-1}}\sum_{I\in\calI_{3,1}}f_{I\setminus\{2j_1-1\}}
\Pi_{I\setminus\{2j_1-1\}}\\ &=(q^{-2j_1-1};q^{-2})_{j_2-j_1}(\udrl{2{j_1}})\binom{j_2}{j_1}_{q^{-2}}\gp{X_{2j_1-1}}f_{2J\setminus
  2\{j_1\}}\Pi_{2(J\setminus\{j_1\})}\cdot\\&\quad\left(1 +
(\udrl{2(m-j_2)})\gp{X_{n-1}}\right)\\ &=f_{2J} (\udrl{2j_1})
\gp{X_{2j_1-1}} \Pi_{2(J\setminus\{j_1\})} \left( 1 +
(\udrl{2(m-j_2)})\gp{X_{n-1}} \right),
\end{align*}
as claimed. If $\eps=1$ we have, using~\eqref{f_I eps=1},
\eqref{equ:S1} for $J\setminus\{j_1\}$, and~\eqref{equ:f eps=1},
\begin{align*}
\calS_{3,1} &=
(q^{-2j_1-1};q^{-2})_{j_2-j_1}(\udrl{2{j_2}+1})\binom{j_2}{j_1}_{q^{-2}}\gp{X_{2j_1}}\sum_{I\in\calI_{3,1}}f_{I\setminus\{2j_1\}}
\Pi_{I\setminus\{2j_1\}}\\ &=(q^{-2j_1-1};q^{-2})_{j_2-j_1}(\udrl{2{j_2}+1})\binom{j_2}{j_1}_{q^{-2}}\gp{X_{2j_1}}f_{2J\setminus
  2\{j_1\}+1}\Pi_{2(J\setminus\{j_1\})+1}\cdot\\&\left(1 +
(\udrl{2(m-j_2)})\gp{X_{n-1}}\right)\\ &=f_{2J+1} (\udrl{2j_1+1})
\gp{X_{2j_1}} \Pi_{2(J\setminus\{j_1\})+1} \left( 1 +
(\udrl{2(m-j_2)})\gp{X_{n-1}} \right),
\end{align*}
as claimed. This establishes~\eqref{equ:S31}.

To prove \eqref{equ:S32} it suffices to observe that for all
$I\in\calI_{3,2}$ we have $f_{I}=f_{I\setminus\{2j_2+\eps-1\}}$.

To prove \eqref{equ:S33} we note that for all $I\in\calI_{3,3}$ we
have

\begin{align}\label{I33 f_I eps=0}
f_{I} &= f_{I\setminus\{2j_1-1\}}
(q^{-2j_1-1};q^{-2})_{j_2-j_1}\binom{j_2}{j_1}_{q^{-2}}\frac{\left(\udrl{2j_1}\right)}{\left(\udrl{2j_2}\right)}\quad\text{
if $\eps=0$},\\
\label{I33 f_I eps=1}
f_{I} &= f_{I\setminus\{2j_1\}}
(q^{-2j_1-1};q^{-2})_{j_2-j_1}\binom{j_2}{j_1}_{q^{-2}}\quad\text{ if $\eps=1$},
\end{align}

Hence, if $\eps=0$ we have, using \eqref{I33 f_I eps=0}, the second
induction hypothesis for the formula~\eqref{equ:S3} for $\calS_{3}$
(for $J\setminus\{j_1\}$, $\eps=0$), and~\eqref{equ:f eps=0},
\begin{align*}
 \calS_{3,3}&=(q^{-2j_1-1};q^{-2})_{j_2-j_1}\binom{j_2}{j_1}_{q^{-2}}\frac{(\udrl{2j_1})}{(\udrl{2j_2})}\gp{X_{2j_1-1}}\sum_{I\in\calI_{3,3}}f_{I\setminus\{2j_1-1\}}\Pi_{I\setminus\{2j_1-1\}}\\
&=(q^{-2j_1-1};q^{-2})_{j_2-j_1}\binom{j_2}{j_1}_{q^{-2}}\frac{(\udrl{2j_1})}{(\udrl{2j_2})}\cdot\\&\quad\gp{X_{2j_1-1}}f_{2J\setminus2\{j_1\}}(\udrl{2j_2})\gp{X_{2j_2-1}}\Pi_{2(J\setminus\{j_1,j_2\})}\left(1+\gp{X_{n-1}}\right)\\
&= f_{2J} (\udrl{2j_1}) \gp{X_{2j_1-1}} \,\gp{X_{2j_2-1}}
\Pi_{2(J\setminus\{j_1,j_2\})} \left( 1 + \gp{X_{n-1}} \right),
\end{align*}
as claimed. If $\eps=1$ we have, using~\eqref{I33 f_I eps=1}, the
second induction hypothesis for the formula~\eqref{equ:S3} for
$\calS_3$ (for $J\setminus\{j_1\}$, $\eps=1$), and~\eqref{equ:f
eps=1},
\begin{align*}
 \calS_{3,3}&=(q^{-2j_1-1};q^{-2})_{j_2-j_1}\binom{j_2}{j_1}_{q^{-2}}\gp{X_{2j_1}}\sum_{I\in\calI_{3,3}}f_{I\setminus\{2j_1\}}\Pi_{I\setminus\{2j_1\}}\\
&=(q^{-2j_1-1};q^{-2})_{j_2-j_1}\binom{j_2}{j_1}_{q^{-2}}\gp{X_{2j_1}}\cdot\\&
\quad f_{2(J\setminus\{j_1\})+1}(\udrl{2j_2+1})\gp{X_{2j_2}}\Pi_{2(J\setminus\{j_1,j_2\})+1}\left(1+\gp{X_{n-1}}\right)\\
&= f_{2J+1} (\udrl{2j_1+1}) \gp{X_{2j_1}} \,\gp{X_{2j_2}}
\Pi_{2(J\setminus\{j_1,j_2\})+1} \left( 1 + \gp{X_{n-1}} \right),
\end{align*}
as claimed. This establishes~\eqref{equ:S33}.

It remains to simplify the sums $\calS_{31}+\calS_{32}+\calS_{33}$ and
$\calS =\calS_{1}+\calS_{2}+\calS_{3}$. For both calculations we use
the following direct consequence of the relations~\eqref{relations}:
for all $j\in\{1+\eps,\dots,m-1\}$ we have
\begin{multline}\label{relations_cons}
\gp{X_{2j+\eps-1}} \left( 1 + \gp{X_{2j+\eps}} + \gp{X_{n-1}} +
(\udrl{2(m-j)})\gp{X_{2j+\eps}}\gp{X_{n-1}}\right) =\\
q^{-2(m-j)}\gp{X_{2j+\eps}}\gp{X_{n-1}}.
\end{multline}
Thus, using \eqref{relations_cons} for~$j=j_2$, we have
\begin{align*}
\lefteqn{\calS_3 = \calS_{3,1} + \calS_{3,2} + \calS_{3,3}} \\&= f_{2J+\eps}
(\udrl{2j_1+\eps}) \gp{X_{2j_1+\eps-1}}
\Pi_{2(J\setminus\{j_1,j_2\})+\eps} \left(\gp{X_{2j_2+\eps}}( 1 +
(\udrl{2(m-j_2)})\gp{X_{n-1}} ) + \right.\\ &\quad \left.\gp{X_{2j_2+\eps-1}} \left(
1 + \gp{X_{2j_2+\eps}} + \gp{X_{n-1}}+
(\udrl{2(m-j_2)})\gp{X_{2j_2+\eps}}\gp{X_{n-1}}\right)\right)\\ &=
f_{2J+\eps} (\udrl{2j_1+\eps}) \gp{X_{2j_1+\eps-1}}
\Pi_{2(J\setminus\{j_1\})+\eps}\cdot\\&\quad\left(1 +
(\udrl{2(m-j_2)})\gp{X_{n-1}} + q^{-2(m-j_2)}\gp{X_{n-1}}\right)\\ &=
f_{2J+\eps} (\udrl{2j_1+\eps}) \gp{X_{2j_1+\eps-1}}
\Pi_{2(J\setminus\{j_1\})+\eps}\left( 1 + \gp{X_{n-1}}\right)
\end{align*}
as claimed in~\eqref{equ:S3}.

It remains to simplify the sum $\calS_1 + \calS_2 + \calS_3$. Using
\eqref{relations_cons} for $j=j_1$, we obtain
\begin{align*}
\lefteqn{\calS = \calS_1 + \calS_2 + \calS_3 = f_{2J+\eps} \Pi_{2(J\setminus\{j_1\})+\eps}
\left( \gp{X_{2j_1+\eps}}( 1 + (\udrl{2(m-j_1)})\gp{X_{n-1}}) +
\right.}\\& \quad\left.(\udrl{2j_1+\eps})\gp{X_{2j_1+\eps-1}}\left( 1 +
\gp{X_{2j_1+\eps}} + \gp{X_{n-1}} + \right.\right.\\&\quad\left.\left.
(\udrl{2(m-j_1)})\gp{X_{2j_1+\eps}}\gp{X_{n-1}}\right)\right)\\ &=f_{2J+\eps}\Pi_{2J+\eps}\left(1 + (\udrl{2(m-j_1)})\gp{X_{n-1}} + (\udrl{2j_1+\eps})q^{-2(m-j_1)}\gp{X_{n-1}}\right)
\\&= f_{2J+\eps}\Pi_{2J+\eps}\left( 1 + (\udrl{n})
\gp{X_{n-1}}\right).
\end{align*}
This concludes the proof of the lemma.
\end{proof}

Given the Euler product~\eqref{equ:euler}, the multiplicative formulae
\eqref{equ:mult H} for $\zeta_{H_{n}(\Gri_\mfp)}(s)$ given in
Theorem~\ref{thmABC:thm B} now follow easily from the following
result.

\begin{proposition}\label{pro:H=F}
For $n=2m+\eps\in\N$ with $\eps\in\{0,1\}$, we have
\begin{equation*}
 \zeta_{H_{n}(\Gri_\mfp)}(s) =
 \frac{\zeta_{K,\mfp}(s-n)}{\zeta_{K,\mfp}(s)}
 \zeta_{F_{m,\eps}(\Gri_\mfp)}(2s-2) =
 \frac{1-q^{-s}}{1-q^{n-s}}\cdot\frac{(q^{2-2s};q^2)_m}{(q^{2(m+\eps)+1-2s};q^2)_m}.
\end{equation*}
\end{proposition}

\begin{proof} The second equality is, of course,
Proposition~\ref{pro:Mult-formula-F}, with $(n,\delta)$ replaced
by~$(m,\eps)$.  We prove the first equality. By
Theorem~\ref{thmABC:thm C}, with $(n,\delta)$ replaced by $(m,\eps)$,
we have
$$\zeta_{F_{m,\eps}(\Gri_\mfp)}(s) = \sum_{J\subseteq[m-1]_0}
f_{F_{m,\eps},J}(q^{-1}) \prod_{j\in J}\gp{X_j(F_{m,\eps})},$$ where
$X_j(F_{m,\eps})= q^{\binom{2m+\eps}{2} - \binom{2j+\eps}{2}-(n-j)s}$.
It follows from Definition~\ref{def:polys thm C} that, for all
$J\subseteq[m-1]_0$,
\begin{equation*}
f_{H_{n},2J+\eps}(q^{-1}) =
\binom{m}{J}_{q^{-2}}(q^{-2(j_1+\eps)-1};q^{-2})_{m-j_1}=
f_{F_{m,\eps},J}(q^{-1})
\end{equation*}
and, for all $j\in[m-1]_0$,
\begin{align*}
X_j(F_{m,\eps})\vert_{s\rightarrow 2s-2} &= q^{\binom{2m+\eps}{2}-\binom{2j+\eps}{2}-(2s-2)(m-j)}\\
&=q^{\binom{2m+\eps}{2}+2m+\eps - \left(\binom{2j+\eps}{2}+2j+\eps\right)-2s(m-j)}\\
&=q^{\binom{2m+1+\eps}{2} - \binom{2j+1+\eps}{2}-2s(m-j)}\\
&=X_{2j+\eps}(H_{n}).
\end{align*}
We thus have, by Theorem~\ref{thmABC:thm C} and
Lemma~\ref{lem:fibres},
\begin{align*}
 \zeta_{H_{n}(\Gri_\mfp)}(s)
 &=\sum_{I\subseteq[n-1]_0}f_{H_{n},I}(q^{-1})\prod_{i\in
 I}\gp{X_i(H_{n})}\\ &=
 \sum_{J\subseteq[m-1]_0}\sum_{I\in\phi^{-1}(J)}f_{H_{n},I}(q^{-1})\prod_{i\in
 I}\gp{X_i(H_{n})}\\ &=
 (1+(\udrl{n})\gp{X_{n-1}(H_{n})})\sum_{J\subseteq[m-1]_0}f_{H_{n,2J+\eps}}(q^{-1})\prod_{j\in
 J}\gp{X_{2j+\eps}(H_{n})}\\ &=
 \frac{1-q^{-s}}{1-q^{n-s}}\sum_{J\subseteq[m-1]_0}
 f_{F_{m,\eps},J}(q^{-1})\prod_{j\in
 J}\gp{X_j(F_{m,\eps})\vert_{s\rightarrow 2s-2}}\\ &=
 \zeta_{K,\mfp}(s-n)\zeta_{K,\mfp}(s)^{-1}
 \zeta_{F_{m,\eps}(\Gri_\mfp)}(2s-2).
\end{align*}
This proves the proposition.
\end{proof}
This concludes the proof of Theorem~\ref{thmABC:thm B}.

\begin{corollary}\label{cor:TypeH-identity}
 For $n=2m+\eps\in\N$ with $\eps\in\{0,1\}$, we have
 \begin{align*}
\sum_{I\subseteq[n-1]_{0}}\left(\prod_{j=1}^{l}(X^{-4};X^{-4})_{\lfloor\mu_{j}/2\rfloor}^{-1}\right)(X^{-2(i_{1}+1)};X^{-2})_{n-i_{1}}\prod_{i\in
I}\mathrm{gp(}(X^{i}Z)^{n-i})\\
=\frac{1-X^{-n-1}Z}{1-X^{n-1}Z}\cdot\frac{(X^{2(-n+1)}Z^{2};X^{4})_{m}}{(X^{\epsilon}Z^{2};X^{4})_{m}}.\nonumber
\end{align*}
\end{corollary}

\begin{proof}
For $X=q^{1/2}$ and $Z=q^{(n+1)/2-s}$ the identity follows from
Theorem~\ref{thmABC:thm C} together with Proposition~\ref{pro:H=F}. As
it holds for infinitely many values of $q$ and $s$, it therefore holds
as a formal identity.

\end{proof}

Corollary~\ref{cor:TypeH-identity} is analogous to
Proposition~\ref{pro:multinomial}. The identity was found by working
backwards from Proposition~\ref{pro:H=F} and Theorem~\ref{thmABC:thm
C}. An independent proof of Corollary~\ref{cor:TypeH-identity} would
yield an alternative proof of Proposition~\ref{pro:H=F}, similar to
the proof of Proposition~\ref{pro:thm B type F and G}.
%
%

\subsection{Another generating function on $B_n$}\label{subsec:distribution L}
We record a corollary of the existence of both `additive' and
`multiplicative' expressions for the local zeta functions of groups of
type $H$, together with Conjecture~\ref{con:L}. Recall that
$n=2m+\epsilon\in\N$, with $\epsilon\in\{0,1\}$.
\begin{proposition}
\label{pro:distribution L}
If Conjecture~\ref{con:L} holds then
 \begin{multline}\label{equ:conjecture}
 \sum_{w\in B_{n}}(-1)^{l(w)}X^{((\sigma+\rmaj)/2-L)(w)}Z^{\rmaj(w)}
   =\\\frac{(1-Z)(X^2Z^2;X^2)_m}{(X^{2(m+\eps)+1}Z^2;X^2)_m}\prod_{i=0}^{n-2}(1-(X^{(n+i+1)/2}Z)^{n-i}).
 \end{multline}
 \end{proposition}

\begin{proof}
Let $\mfp$ be a non-zero prime ideal $\mfp$ of $\Gri$, with
$|\Gri:\mfp|=q$, say. By Theorem~\ref{thmABC:thm C} and
Proposition~\ref{pro:H=F} we have
\begin{multline*}
\zeta_{H_{n}(\Gri_\mfp)}(s)
=\sum_{I\subseteq[n-1]_{0}}f_{H_{n},I}(q^{-1})\prod_{i\in
I}\gp{q^{\binom{n+1}{2}-\binom{i+1}{2}-(n-i)s}}
=\\\frac{(1-q^{-s})}{(1-q^{n-s})}\frac{(q^{2-2s};q^2)_m}{(q^{2(m+\eps)+1-2s};q^2)_m},
\end{multline*}
and, assuming Conjecture~\ref{con:L}, we have\[
f_{H_{n},I}(q^{-1})=\sum_{w\in B_{n}^{I^{c}}}(-1)^{l(w)}q^{-L(w)}.\]
Therefore Lemma~\ref{lem:Common-denominator} with
$(W,S)=(B_{n},\{s_0,\dots,s_{n-1}\})$  implies that
\[
\zeta_{H_{n}(\Gri_\mfp)}(s)=\frac{\sum_{w\in
    B_{n}}(-1)^{l(w)}q^{-L(w)}\prod_{i\in
    D(w)}q^{((n+i+1)/2-s)(n-i)}}{\prod_{i=0}^{n-1}(1-q^{((n+i+1)/2-s)(n-i)})},
\]
and so
\begin{align*}
 \lefteqn{\sum_{w\in
  B_{n}}(-1)^{l(w)}q^{((\sigma+\rmaj)/2-L)(w)}(q^{-s})^{\rmaj(w)}}&\\
  & =\sum_{w\in B_{n}}(-1)^{l(w)}q^{-L(w)+\sum_{i\in
  D(w)}((n+i)(n-i)+n-i)/2}(q^{-s})^{\sum_{i\in D(w)}(n-i)}\\ & =\sum_{w\in
  B_{n}}(-1)^{l(w)}q^{-L(w)}\prod_{i\in D(w)}q^{((n+i+1)/2-s)(n-i)}\\
  &
  =\frac{1-q^{-s}}{1-q^{n-s}}\cdot\frac{(q^{2-2s};q^2)_m}{(q^{2(m+\eps)+1-2s};q^2)_m}\cdot\prod_{i=0}^{n-1}(1-q^{((n+i+1)/2-s)(n-i)})\\
  &=
  \frac{(1-q^{-s})(q^{2-2s};q^2)_m}{(q^{2(m+\eps)+1-2s};q^2)_m}\prod_{i=0}^{n-2}(1-q^{((n+i+1)/2-s)(n-i)}).
\end{align*}
This identity holds for infinitely many values of $q$ and $s$, and
hence yields a formal identity in variables $X=q$ and $Z=q^{-s}$.

\end{proof}

\subsection{Proof of Corollary~\ref{cor:thm B}}\label{subsec:proof cor thm B}
The functional equations in~\eqref{item:funeqs} follow directly from
the formulae given in Theorem~\ref{thmABC:thm B}, and the Euler
product for the Dedekind zeta function~$\zeta_K(s)$;
cf.~\eqref{equ:dedekind}. The abscissae of convergence
in~\eqref{item:alphas} and the analytical statements
in~\eqref{item:analytic} reflect classical facts about $\zeta_K(s)$,
viz.\ its abscissa of convergence $1$ and meromorphic continuation to
the whole complex plane with a simple pole at~$s=1$. The asymptotic
statements in~\eqref{item:asymptotic} follow from standard Tauberian
theorems; cf., for instance, \cite[Theorem 4.20]{duSG/00}.

\subsection{Jordan's totient functions}\label{subsec:jordan}
We record an interpretation of the zeta functions of groups of type
$F$, $G$ and $H$, described in Theorem~\ref{thmABC:thm B}, in terms of
Jordan's totient functions. Given $b, n\in\N$, let $J_b(n)$ be the
number of $b$-tuples $(a_1,a_2,\dots,a_b)$ of integers satisfying
$1\leq a_i \leq n$, for all $i$, such that
$\gcd(a_1,\dots,a_b,n)=1$. The function $J_b$ is called the $b$-th
Jordan totient function; cf.~\cite[1.5.2]{Murty/08}. Clearly
$J_1=\phi$, the Euler totient function.

\begin{lemma}\label{lem:jordan}
 Let $a\in\N_0$, $b\in\N$. The Dirichlet generating series for the
 arithmetic function $J_{a,b}:\N \rightarrow \N, n\mapsto n^a
 J_b(n)$ is $\sum_{n=1}^\infty J_{a,b}(n)n^{-s} = \zeta(s-a-b) /
 \zeta(s-a),$ where $\zeta(s)$ is the Riemann zeta function.
\end{lemma}

\begin{proof}
This follows easily from the fact that, for a prime power $p^e$,
$e\in\N$, one has $J_b(p^e)=p^{eb}(1-p^{-b})$.
\end{proof}

We observe that, for $\bfG\in\{F_{n,\delta},G_{n}, H_{n}\}$ the zeta
function of the group $\bfG(\Z)$ is a finite product of factors of the
form $\zeta(s-a-b) / \zeta(s-a)$, for suitable pairs of
integers~$(a,b)$. By Lemma~\ref{lem:jordan}, the arithmetic function
$n\mapsto \widetilde{r}_{\bfG(\Z)}(n)$ may thus be described as a finite
convolution product of functions of the form $J_{a,b}$. Over number
rings, one may define analogues to the functions $J_{a,b}$ in terms of
tuples of coprime ideals, and hence obtain analogous expressions for
the zeta functions of groups of the form $\bfG(\Gri)$ as suitable
convolution products.

\section{Analogy with prehomogeneous vector spaces}\label{sec:pvs}

The local factors of representation zeta functions of groups of type
$F, G$ and $H$ studied in this paper bear a striking resemblance to
the zeta integrals associated to Igusa local zeta functions of certain
prehomogeneous vector spaces (PVS).

An irreducible PVS is a pair $(V,G)$, comprising an
$n$-di\-men\-sio\-nal complex vector space $V$ and a connected
algebraic subgroup $G$ of $\GL(V)$, acting irreducibly on $V$ with a
Zariski-dense $G$-orbit. Irreducible PVS were classified, up to a
certain equivalence relation, by Kimura and Sato in terms of
irreducible, so-called reduced PVS. Associated with an irreducible
reduced PVS $(G,V)$ there is a `relatively invariant' irreducible
polynomial $f(\bfx)\in\C[x_1,\dots,x_n]$ with the property that
$V\setminus f^{-1}(0)$ is the Zariski-dense $G$-orbit in~$V$. If
$f\in\lri[x_1,\dots,x_n]$ for a compact discrete valuation ring $\lri$
of characteristic zero then the integral
$$Z^\lri_f(s) = \int_{\lri^n} \vert f(x) \vert ^{s}\tud\mu$$ is known
as Igusa's local zeta function attached to~$f$. The real parts of the
poles of $Z^\lri_f(s)$ are known to be among the zeros of the
Bernstein-Sato polynomial $b_f(s)$ associated to~$f$. The polynomial
$b_f(s)$ provides a measure of the complexity of the singularities of
the hypersurface~$f^{-1}(0)$. This intriguing interpretation of the
real parts of poles of the integral $Z^\lri_f(s)$ is conjectured to
hold in a much more general context: If $f(x)\in K[x_1,\dots,x_n]$,
where $K$ is a number field, the Bernstein-Sato polynomial conjecture
states that, for almost all non-archimedean completions $\lfi$ of
$\Gfi$ with valuation ring~$\lri$, say, the real parts of the poles of
$Z^\lri_f(s)$ should be among the zeros of $b_f(s)$;
cf.~\cite{Kimura/03, Igusa/00} for further details on PVS, and
\cite[Section~7]{Denef/91} for details on the Bernstein-Sato
polynomial conjecture.

The list of irreducible reduced PVS in the appendix of
\cite{Kimura/03} starts off with three infinite families of
prehomogeneous vector spaces of generic matrices, viz.\ the vector
spaces $\Mat_n(\C), \Sym_n(\C)$ and $\Alt_{2n}(\C)$, respectively. The
associated relative invariants are $f(X)=\det(X), f(X)=\det(X)$ and
$f(X)=\Pfaff(X)$, respectively, where $\Pfaff(X)$ denotes the Pfaffian
of an antisymmetric matrix~$X$. Let $\lri$ be a complete discrete
valuation ring with residue field cardinality $q$. We continue to
write $n=2m+\varepsilon\in\N$ with $\varepsilon\in\{0,1\}$. The
following formulae for the associated Igusa zeta functions are well
known; see, for instance,~\cite[pp.\ 164, 163, 177]{Igusa/00}.

\begin{align}
 Z_{\Alt_{2n}(\lri)}(s) &:=
\int_{X\in\Alt_{2n}(\lri)}\vert\Pfaff(X)\vert^s\tud\mu =
\prod_{i=0}^{n-1} \frac{1-q^{-1-2i}}{1-q^{-s-1-2i}},
\label{pvs alt}\\ Z_{\Mat_{n}(\lri)}(s) &:=
\int_{X\in\Mat_{n}(\lri)}\vert\det(X)\vert^s\tud\mu =
\prod_{i=0}^{n-1}\frac{1-q^{-1-i}}{1-q^{-s-1-i}},\label{pvs mat}\\
Z_{\Sym_n(\lri)}(s) &:=
\int_{X\in\Sym_{n}(\lri)}\vert\det(X)\vert^s\tud\mu =
\frac{1-q^{-(1-\eps)(s+1)-n}}{1-q^{-s-1}}\prod_{i=0}^{m-1}\frac{1-q^{-1-2i}}{1-q^{-2s-3-2i}}.
\label{pvs sym}
\end{align}

We observe that, in analogy to Proposition~\ref{pro:H=F}, we have
\begin{equation}\label{equ:relation pfx}
Z_{\Sym_{2n}(\lri)}(s) = \frac{1-q^{-s-2n-1}}{1-q^{-s-1}} Z_{\Alt_{2n}(\lri)}(2s+2).
\end{equation}
We are not aware of a conceptional explanation for this identity.

\medskip
The group schemes $F_{n,\delta}$, $G_{n}$ and $H_{n}$ studied in this
paper are designed so that the commutator matrices associated to their
respective Lie rings reflect the prehomogeneous vector spaces
$\Alt_{2n}(\C)$, $\Mat_n(\C)$ and $\Sym_n(\C)$,
respectively. Theorem~\ref{thmABC:thm B} shows that the local
representation zeta functions associated to groups of type $F$, $G$
and $H$ closely resemble the $\mfp$-adic integrals~\eqref{pvs alt},
\eqref{pvs mat} and \eqref{pvs sym}, without being obtainable from
these integrals by simple transformations of variables. We record an
immediate consequence of Theorem~\ref{thmABC:thm B} regarding the
poles of the local zeta functions.
\begin{corollary}\label{cor:poles}
 Let $\mathbf{G}\in\{F_{n,\delta},G_{n},H_{n}\}$. There exists a
 finite set $P(\mathbf{G})$ of rational numbers such that the
 following holds. Given a ring of integers $\Gri$ of a number
 field~$K$, and for any non-zero prime ideal $\mfp$ of $\Gri$, we
 have \[ P(\mathbf{G})=\{\real(s)\mid s\in\C\text{ a pole of
 }\zeta_{\mathbf{G}(\Gri_\mfp)}(s)\}.\] More precisely, we have
\begin{align*}
P(F_{n,\delta}) & =\{2(n+i+\delta)-1\mid i\in[n-1]_{0}\},\\
P(G_{n}) & =\{n+i\mid i\in[n-1]_{0}\},\\
P(H_{n}) & =\{n,m+i+\varepsilon+1/2\mid i\in[m-1]_{0}\}.\end{align*}
\end{corollary}
In other words, the set of real parts of poles of
$\zeta_{\mathbf{G}(\Gri_\mfp)}(s)$ is independent of $\mathcal{O}$
and~$\mfp$.  We note further that the sets $P(\bfG)$ defined in
Corollary~\ref{cor:poles} are obtained from the sets of real parts of
the poles of $\mfp$-adic integral associated to the corresponding PVS
by translation by the global abscissa of convergence of the relevant
representation zeta function. Indeed we have, by \eqref{pvs alt},
\eqref{pvs mat} and~\eqref{pvs sym}, Corollary~\ref{cor:poles} and
\eqref{equ:alphas}, that for all $\lri$
\begin{align*}
\{\real(s) \mid s\in\C \text{ a pole of
}Z_{\Alt_{2n}(\lri)}(s)\}&=P(F_{n,\delta})-\alpha(F_{n,\delta}),\\ \{\real(s) \mid s\in\C
\text{ a pole of
}Z_{\Mat_{n}(\lri)}(s)\}&=P(G_{n})-\alpha(G_{n}),\\ \{\real(s) \mid
s\in\C \text{ a pole of
}Z_{\Sym_n(\lri)}(s)\}&=P(H_{n})-\alpha(H_{n}).\\
\end{align*}

As mentioned above, the zeta integrals~\eqref{pvs sym}, \eqref{pvs
  mat} and \eqref{pvs sym} are examples of Igusa zeta functions which
are known to satisfy the Bernstein-Sato-polynomial conjecture. More
precisely, they are Igusa zeta functions with the property that the
real parts of the poles of $Z(s)$ are among the zeros of the
Bernstein-Sato polynomial $b_f(s)$, with pole multiplicities not
exceeding the multiplicities of the corresponding zeros. In the
cases~\eqref{pvs alt} and \eqref{pvs mat} these sets coincide. In the
case~\eqref{pvs sym}, however, the set of Bernstein-Sato zeros is
strictly larger: indeed, in this case we have
$b_f(s)=\prod_{i=0}^{n-1}(s+(i+2)/2)$. It strikes us as remarkable
that
$$\{s\in\C \mid b_f(s)=0\} = \{a(H_{n},i)/(n-i) \mid i\in[n-1]_0\} -
\alpha(H_{n}).$$ The rational numbers $a(H_{n},i)/(n-i)$ arise as
candidate real parts coming from terms of the form $X_i/(1-X_i)$, with
$X_i=q^{a(H_{n},i)-(n-i)s}$. We remark that the numbers that do not
give rise to real parts of poles come from the variables which cancel
in the transition from the additive to the multiplicative formulae;
cf.\ Lemma~\ref{lem:fibres}.

It would be interesting to associate other irreducible prehomogeneous
vector spaces with finitely generated nilpotent groups. We are
presently not aware of any other irreducible reduced prehomogeneous
vector spaces whose geometry is reflected in this way in
representation zeta functions of $\T$-groups.

\begin{acknowledgements}
  This research was supported by EPSRC grant EP/F044194/1. We are
  grateful to the referee who pointed out a mistake in an earlier
  version of this paper, and to Tobias Ro\ss mann who helped us fix
  it.
\end{acknowledgements}


\providecommand{\bysame}{\leavevmode\hbox to3em{\hrulefill}\thinspace}
\providecommand{\MR}{\relax\ifhmode\unskip\space\fi MR }
\providecommand{\MRhref}[2]{%
  \href{http://www.ams.org/mathscinet-getitem?mr=#1}{#2}
}
\providecommand{\href}[2]{#2}

\end{document}